\documentclass[12pt]{article}

\usepackage{amsfonts, amssymb, amsmath, amsthm, epsf}

\makeatletter

\@addtoreset{figure}{section}

\theoremstyle{plain}
\newtheorem{theorem}{Theorem}
\@addtoreset{theorem}{section}

\newtheorem{lemma}{Lemma}
\@addtoreset{lemma}{section}

\newtheorem{corollary}{Corollary}
\@addtoreset{corollary}{section}

\newtheorem{condition}{Condition}
\@addtoreset{condition}{section}

\theoremstyle{definition}

\newtheorem{definition}{Definition}
\@addtoreset{definition}{section}

\@addtoreset{example}{section}

\newtheorem{remark}{Remark}
\@addtoreset{remark}{section}

\@addtoreset{equation}{section}

\makeatother

\renewcommand{\Im}{{\rm Im\,}}

\newcommand{\ind}{{\rm ind\,}}
\newcommand{\supp}{{\rm supp\,}}
\newcommand{\Dom}{{\rm D\,}}
\renewcommand{\ker}{{\rm ker\,}}
\renewcommand{\dim}{{\rm dim\,}}
\newcommand{\codim}{{\rm codim\,}}
\newcommand{\Gr}{{\rm Gr\,}}

\newcommand{\Orb}{{\rm Orb}}

\newcommand{\const}{{\rm const}}
\newcommand{\dist}{{\rm dist}}

\title{On the Stability of the Index of Unbounded Nonlocal Operators in Sobolev Spaces}
\author{Pavel Gurevich\thanks{Supported by the
Russian Foundation for Basic Research (project No.~04-01-00256).}
\thanks{E-mail: gurevichp@gmail.com}}

\date{}

\begin{document}

\maketitle


\begin{abstract}
Unbounded operators corresponding to nonlocal elliptic problems on
a bounded region $G\subset\mathbb R^2$ are considered. The domain
of these operators consists of functions from the Sobolev space
$W_2^m(G)$ being generalized solutions of the corresponding
$2m$-order elliptic equation with right-hand side from $L_2(G)$
and satisfying homogeneous nonlocal boundary conditions. It is
known that such unbounded operators have the Fredholm property. It
is proved in the paper that low-order terms in the differential
equation do not affect the index of the operator. Conditions under
which nonlocal perturbations on the boundary do not change the
index are also formulated.
\end{abstract}

\section*{Introduction}
\addcontentsline{toc}{section}{\protect\numberline{}Introduction}

In the one-dimensional case, nonlocal problems were studied by
A.~Sommerfeld~\cite{Sommerfeld}, J.~D.~Tamarkin~\cite{Tamarkin},
M.~Picone~\cite{Picone}. T.~Carleman~\cite{Carleman} considered
the problem of finding a function harmonic on a two-dimensional
bounded domain and subjected to a nonlocal condition connecting
the values of this function at different points of the boundary.
A.~V.~Bitsadze and A.~A.~Smarskii~\cite{BitzSam} suggested another
setting of a nonlocal problem arising in plasma theory: to find a
function harmonic on a bounded domain and satisfying nonlocal
conditions on shifts of the boundary that can take points of the
boundary inside the domain. Different generalizations of the above
nonlocal problems were investigated by many
authors (see~\cite{SkBook} and references therein).

It turns out that the most difficult situation occurs if the
support of nonlocal terms intersects the boundary. In this case,
solutions of nonlocal problems can have power-law singularities
near some points even if the boundary and the right-hand sides are
infinitely smooth~\cite{SkMs86}. For this reason, such
problems are naturally studied in weighted spaces (introduced by
V.~A.~Kondrat'ev for boundary-value problems in nonsmooth
domains~\cite{KondrTMMO67}). The most complete theory of nonlocal
problems in weighted spaces is developed by
A.~L.~Skubachevskii~\cite{SkMs86,SkDu90,SkDu91,SkJMAA,SkBook} and his pupils.

Note that the study of nonlocal problems is motivated both by
significant theoretical progress in that direction and important
applications arising in biophysics, theory of diffusion
processes, plasma theory, and so on.

In this paper, we investigate the influence of low-order terms in
the elliptic equation and the influence of nonlocal perturbations
in boundary conditions upon the index of the unbounded nonlocal
operator in $L_2(G)$. This issue was earlier studied by A.~L.
Skubachevskii~\cite{SkJMAA} for bounded operators in weighted
spaces. It is proved in~\cite{SkJMAA} that nonlocal perturbations
supported outside the points of conjugation of boundary conditions
do not change the index of the corresponding bounded operator. The
similar assertion has later been established in Sobolev spaces in
the two-dimensional case~\cite{GurRJMP04}. In both cases, one can
either use the method of continuation with respect to parameter or
reduce the original problem to that where nonlocal perturbations
have compact square. As for low-order terms in the elliptic
equation, they are simply compact perturbations.

The situation is quite different in the case of unbounded
operators. The difficulty is that the low-order terms in elliptic
equations are not compact or relatively compact (see
Definition~\ref{defPCompact}); moreover, if the order of the
elliptic equation is greater than two, they are not even
relatively bounded, and, therefore, they change the domain of
definition of the operator. As for nonlocal perturbations in
boundary conditions, they explicitly change the domain of
definition, and, therefore, they cannot be regarded as compact
perturbations (in any sense) either.

To overcome the above difficulties, we consider an auxiliary
operator (whose index equals the index of the original operator)
acting on weighted spaces. In Sec.~\ref{secPertLOT}, we prove that
low-order terms in elliptic equations are relatively compact
perturbations  of the auxiliary operator, and, therefore, they do
not affect the index. In Sec.~\ref{secPertNC}, we consider
nonlocal perturbations in boundary conditions, which explicitly
change the domain of definition. We make use of the notion of a
\textit{gap between unbounded operators} (see
Definition~\ref{defGap}). We show that, if nonlocal perturbations
in boundary conditions satisfy some regularity conditions at the
conjugation points, then multiplying the perturbations by a small
parameter leads to a small gap between the corresponding
operators. Combining this fact with the method of continuation
with respect to parameter, we prove the index stability theorem.

Finally, we note that the Fredholm property of unbounded nonlocal
operators on $L_2(G)$ was earlier studied either for the case in
which nonlocal conditions were set on shifts of the
boundary~\cite{SkBook} or in the case of a nonlocal perturbation
of the Dirichlet problem for a second-order elliptic
equation~\cite{GM,Gusch}. Elliptic equations of order $2m$ with
general nonlocal conditions are being investigated for the first
time.

\section{Setting of Nonlocal Problems in Bounded Domains}\label{secSetting}

\subsection{Setting of nonlocal problems}\label{subsecSetting}
Let $G\subset{\mathbb R}^2$ be a bounded domain with boundary
$\partial G$. We introduce a set ${\mathcal K}\subset\partial G$
consisting of finitely many points and assume that $\partial
G\setminus{\mathcal K}=\bigcup_{i=1}^{N}\Gamma_i$, where
$\Gamma_i$ are open (in the topology of $\partial G$)
$C^\infty$-curves. In a neighborhood of each point $g\in{\mathcal
K}$, the domain $G$ is supposed to coincide with some plane angle.

For any domain $Q$ and for integer $k\ge0$, we denote by
$W^k(Q)=W_2^k(Q)$ the Sobolev space with the norm
$$
\|u\|_{W^k(Q)}=\left(\sum\limits_{|\alpha|\le k}\int\limits_Q
|D^\alpha u|^2\,dy\right)^{1/2}
$$
(we set $W^0(Q)=L_2(Q)$ for $k=0$). For an integer $k\ge1$, we
introduce the space $W^{k-1/2}(\Gamma)$ of traces on a smooth
curve $\Gamma\subset\overline{Q}$, with the norm
\begin{equation}\label{eqTraceNormW}
\|\psi\|_{W^{k-1/2}(\Gamma)}=\inf\|u\|_{W^k(Q)}\quad (u\in
W^k(Q):\ u|_\Gamma=\psi).
\end{equation}

For any set $X\in \mathbb R^2$ having a nonempty interior, we
denote by $C_0^\infty(X)$ the set of functions infinitely
differentiable on $\overline{ X}$ and compactly supported on $X$.

Now we introduce different weighted spaces for different domains
$Q$. Consider the following cases: 1.~$Q=G$; denote $\mathcal
M=\mathcal K$; 2.~$Q$ is a plane angle $K=\{y\in\mathbb R^2:\
|\omega|<\omega_0\}$, where $0<\omega_0<\pi$; denote $\mathcal
M=\{0\}$; 3.~$Q=\{y\in\mathbb R^2:\ |\omega|<\omega_0,\
0<r<\varepsilon\}$ for some $\varepsilon>0$; denote $\mathcal
M=\{0\}$. Introduce the weighted Kondrat'ev space
$H_a^k(Q)=H_a^k(Q,\mathcal M)$ as the completion of the set
$C_0^\infty(\overline{Q}\setminus\mathcal M)$ with respect to the
norm
$$
\|u\|_{H_a^k(Q)}=\bigg(\sum\limits_{|\alpha|\le s}\int\limits_Q
\rho^{2(a+|\alpha|-k)}|D^\alpha u|^2\, dx\bigg)^{1/2},
$$
where $k\ge0$, $a\in\mathbb R$, and $\rho(y)=\dist(y,\mathcal M)$;
clearly, $\rho(y)=r$ in Cases~2 and~3 ($r$ being the polar
radius).

Denote by $H_a^{k-1/2}(Q)$ ($k\ge1$ is an integer) the space of
traces on a smooth curve $\Gamma\subset\overline Q$, with the norm
$$
\|\psi\|_{H_a^{k-1/2}(\Gamma)}=\inf\|v\|_{H_a^k(Q)}\qquad (v\in
H_a^k(Q):\ v|_\Gamma=\psi).
$$

We denote by ${\bf A}(y, D_y)$ and $B_{i\mu s}(y, D_y)$
differential operators of order $2m$ and $m_{i\mu}$ ($m_{i\mu}\le
m-1$), respectively, with complex-valued $C^\infty$-coefficients
($i=1, \dots, N;$ $\mu=1, \dots, m;$ $s=0, \dots, S_i$). In
particular, we set $ \mathbf B_{i\mu}^0u=B_{i\mu 0}(y,
D_y)u|_{\Gamma_i}. $

\begin{condition}[cf., e.g.,~\cite{LM}]\label{condRegLocalProb}
The operator ${\bf A}(y, D_y)$ is properly elliptic for all $y\in
\overline{G}$, and the system of operators $\{\mathbf
B_{i\mu}^0\}_{\mu=1}^m$ covers ${\bf A}(y, D_y)$ for all $i=1,
\dots, N$ and $y\in\overline{\Gamma_i}$.
\end{condition}
The operators ${\bf A}(y, D_y)$ and $\mathbf B_{i\mu}^0$ will
\textit{correspond to a ``local'' boundary-value problem}.

\smallskip

Now we define operators corresponding to nonlocal conditions near
the set $\mathcal K$. For ${\varepsilon}>0$ and any closed set
$\mathcal N$, denote by $\mathcal O_{\varepsilon}(\mathcal
N)=\{y\in \mathbb R^2: \dist(y, \mathcal N)<\varepsilon\}$ its
$\varepsilon$-neighborhood.

Let $\Omega_{is}$ ($i=1, \dots, N;$ $s=1, \dots, S_i$) be
$C^\infty$-diffeomorphisms taking some neighborhood ${\mathcal
O}_i$ of the curve $\overline{\Gamma_i\cap\mathcal
O_{{\varepsilon}}(\mathcal K)}$ to the set $\Omega_{is}({\mathcal
O}_i)$ in such a way that $ \Omega_{is}(\Gamma_i\cap\mathcal
O_{{\varepsilon}}(\mathcal K))\subset G $ and
$\Omega_{is}(g)\in\mathcal K$ for $
g\in\overline{\Gamma_i}\cap\mathcal K.$
 Thus, under
the transformations $\Omega_{is}$, the curves
$\Gamma_i\cap\mathcal O_{{\varepsilon}}(\mathcal K)$ are mapped
strictly inside the domain $G$, whereas the set of end points
$\overline{\Gamma_i}\cap\mathcal K$ is mapped to itself.

Let us specify the structure of the transformations $\Omega_{is}$
near the set $\mathcal K$. Denote by the symbol $\Omega_{is}^{+1}$
the transformation $\Omega_{is}:{\mathcal
O}_i\to\Omega_{is}({\mathcal O}_i)$ and by $\Omega_{is}^{-1}$ the
inverse transformation. The set of all points $
\Omega_{i_qs_q}^{\pm1}(\dots\Omega_{i_1s_1}^{\pm1}(g))\in{\mathcal
K}\ (1\le s_j\le S_{i_j},\ j=1, \dots, q), $ i.e., the set of all
points that can be obtained by consecutively applying the
transformations $\Omega_{i_js_j}^{+1}$ or $\Omega_{i_js_j}^{-1}$
(taking the points of ${\mathcal K}$ to ${\mathcal K}$) to the
point $g\in\mathcal K$, is called an {\it orbit} of the point $g$
and is denoted by $\Orb(g)$.

Clearly, for any $g, g'\in{\mathcal K}$ either $\Orb(g)=\Orb(g')$
or $\Orb(g)\cap\Orb(g')=\varnothing$. In what follows, we assume
that the set $\mathcal K$ consists of one orbit. (All results can
be directly generalized to the case in which $\mathcal K$ consists
of finitely many mutually disjoint orbits, see
Remark~\ref{remManyOrbits}.) Denote the points of the set (orbit)
$\mathcal K$ by $g_j$, $j=1, \dots, N$.

Take a small number $\varepsilon$ such that there exist
neighborhoods $\mathcal O_{\varepsilon_1}(g_j)$ of the points
$g_j\in\mathcal K$ satisfying the following conditions: 1.~$
\mathcal O_{\varepsilon_1}(g_j)\supset\mathcal
O_{\varepsilon}(g_j) $; 2.~the boundary $\partial G$ coincides
with some plane angle in the neighborhood $\mathcal
O_{\varepsilon_1}(g_j)$; 3.~$\overline{\mathcal
O_{\varepsilon_1}(g_j)}\cap\overline{\mathcal
O_{\varepsilon_1}(g_k)}=\varnothing$ for any $g_j,g_k\in\mathcal
K$, $j\ne k$; 4.~if $g_j\in\overline{\Gamma_i}$ and
$\Omega_{is}(g_j)=g_k,$ then ${\mathcal
O}_{\varepsilon}(g_j)\subset\mathcal
 O_i$ and
 $\Omega_{is}\big({\mathcal
O}_{\varepsilon}(g_j)\big)\subset{\mathcal
O}_{\varepsilon_1}(g_k).$

For each point $g_j\in\overline{\Gamma_i}\cap\mathcal K$, we fix a
transformation $Y_j: y\mapsto y'(g_j)$ which is the composition of
the shift by the vector $-\overrightarrow{Og_j}$ and the rotation
through some angle so that
$$
Y_j({\mathcal O}_{\varepsilon_1}(g_j))={\mathcal
O}_{\varepsilon_1}(0),\qquad Y_j(G\cap{\mathcal
O}_{\varepsilon_1}(g_j))=K_j\cap{\mathcal O}_{\varepsilon_1}(0),
$$
$$
Y_j(\Gamma_i\cap{\mathcal
O}_{\varepsilon_1}(g_j))=\gamma_{j\sigma}\cap{\mathcal
O}_{\varepsilon_1}(0)\quad (\sigma=1\ \text{or}\ \sigma=2),
$$
where $ K_j=\{y\in{\mathbb R}^2:\ r>0,\ |\omega|<\omega_j\},$ $
\gamma_{j\sigma}=\{y\in\mathbb R^2:\ r>0,\ \omega=(-1)^\sigma
\omega_j\},$ $(\omega,r)$ are the polar coordinates, and
$0<\omega_j<\pi$.

\begin{condition}\label{condK1}
Let $g_j\in\overline{\Gamma_i}\cap\mathcal K$ and
$\Omega_{is}(g_j)=g_k\in\mathcal K;$ then the transformation $
Y_k\circ\Omega_{is}\circ Y_j^{-1}:{\mathcal
O}_{\varepsilon}(0)\to{\mathcal O}_{\varepsilon_1}(0) $ is the
composition of a rotation and a homothety.
\end{condition}

\begin{remark}\label{remK1}
In particular, Condition~\ref{condK1}, being combined with the
assumption $\Omega_{is}(\Gamma_i\cap\mathcal
O_\varepsilon(\mathcal K))\subset G$, means that, if
$g\in\Omega_{is}(\overline{\Gamma_i}\cap\mathcal
K)\cap\overline{\Gamma_j}\cap{\mathcal K}\ne\varnothing$, then the
curves $\Omega_{is}(\overline{\Gamma_i})$ and
$\overline{\Gamma_j}$ are not tangent to each other at the point
$g$.
\end{remark}

Consider a number $\varepsilon_0$,
$0<\varepsilon_0\le\varepsilon$, satisfying the following
condition: if $g_j\in\overline{\Gamma_i}$ and
$\Omega_{is}(g_j)=g_k,$ then ${\mathcal
O}_{\varepsilon_0}(g_k)\subset \Omega_{is}\big({\mathcal
O}_{\varepsilon}(g_j)\big)$. Introduce a function $\zeta\in
C^\infty(\mathbb R^2)$ such that $ \zeta(y)=1$  for $y\in\mathcal
O_{\varepsilon_0/2}(\mathcal K)$ and $\supp\zeta\subset\mathcal
O_{\varepsilon_0}(\mathcal K)$.

Now we define nonlocal operators $\mathbf B_{i\mu}^1$ by the
formula
$$
 \mathbf B_{i\mu}^1u=\sum\limits_{s=1}^{S_i}
   \big(B_{i\mu s}(y,
   D_y)(\zeta u)\big)\big(\Omega_{is}(y)\big),\
   y\in\Gamma_i\cap\mathcal O_{\varepsilon}(\mathcal K),\qquad
 \mathbf B_{i\mu}^1u=0,\
y\in\Gamma_i\setminus\mathcal O_{\varepsilon}(\mathcal K),
$$
where $\big(B_{i\mu s}(y,
D_y)u\big)\big(\Omega_{is}(y)\big)=B_{i\mu s}(x,
D_{x})u(x)|_{x=\Omega_{is}(y)}$. Since $\mathbf B_{i\mu}^1u=0$
whenever $\supp u\subset\overline{ G}\setminus\overline{\mathcal
O_{\varepsilon_0}(\mathcal K)}$, we say that the operators
$\mathbf B_{i\mu}^1$ \textit{correspond to nonlocal terms
supported near the set} $\mathcal K$.

\smallskip

For any $\rho>0$, we denote $G_\rho=\{y\in G: \dist(y,
\partial G)>\rho\}$. Consider operators $\mathbf
B_{i\mu}^2$ satisfying the following condition
(cf.~\cite{SkMs86,SkJMAA,GurRJMP03}).
\begin{condition}\label{condSeparK23}
There exist numbers $\varkappa_1>\varkappa_2>0$ and $\rho>0$ such
that the following inequalities hold{\rm :}
\begin{equation}\label{eqSeparK23'}
  \|\mathbf B^2_{i\mu}u\|_{W^{2m-m_{i\mu}-1/2}(\Gamma_i)}\le c_1
  \|u\|_{W^{2m}(G\setminus\overline{\mathcal O_{\varkappa_1}(\mathcal
  K)})},
\end{equation}
\begin{equation}\label{eqSeparK23''}
  \|\mathbf B^2_{i\mu}u\|_{W^{2m-m_{i\mu}-1/2}
   (\Gamma_i\setminus\overline{\mathcal O_{\varkappa_2}(\mathcal K)})}\le
  c_2 \|u\|_{W^{2m}(G_\rho)}.
\end{equation}
\end{condition}

\begin{remark}
In~\eqref{eqSeparK23'}, \eqref{eqSeparK23''}, and throughout the
paper, we denote by $c,c_1,c_2,\dots$ and $k_1,k_2,\dots$ positive
constants which do not depend on the functions entering the
corresponding inequality.
\end{remark}

We assume that
Conditions~\ref{condRegLocalProb}--\ref{condSeparK23} hold
throughout, including the formulation of lemmas.

It follows from~\eqref{eqSeparK23'} that $\mathbf B_{i\mu}^2u=0$
whenever $\supp u\subset \mathcal O_{\varkappa_1}(\mathcal K)$.
For this reason, we say that the operators $\mathbf B_{i\mu}^2$
\textit{correspond to nonlocal terms supported outside the set}
$\mathcal K$.

We study the following nonlocal elliptic problem:
\begin{align}
 {\bf A}(y, D_y)u=f(y) \quad &(y\in G),\label{eqPinG}\\
     \mathbf B_{i\mu}u\equiv\mathbf B_{i\mu}^0 u+\mathbf B_{i\mu}^1 u+\mathbf B_{i\mu}^2 u=
   0\quad
    &(y\in \Gamma_i;\ i=1, \dots, N;\ \mu=1, \dots, m),\label{eqBinG}
\end{align}
where $f\in L_2(G)$. Introduce the space $W^m(G,\mathbf B)$
consisting of functions $u\in W^m(G)$ that satisfy homogeneous
nonlocal conditions~(\ref{eqBinG}). Consider the unbounded
operator $\mathbf P: \Dom(\mathbf P)\subset L_2(G)\to L_2(G)$
given by
$$
 \mathbf P u=\mathbf A(y, D_y)u,\qquad u\in \Dom(\mathbf P)=\{u\in
 W^m(G,\mathbf B):\ \mathbf A(y, D_y) u\in L_2(G)\}.
$$

\begin{definition}\label{defGenSol1}
A function $u$ is called a {\it generalized solution} of
problem~(\ref{eqPinG}), (\ref{eqBinG}) with right-hand side $f\in
L_2(G)$ if $u\in \Dom(\mathbf P)$ and $
 \mathbf P u=f.
$
\end{definition}
Equivalent definition of a generalized solution can be given in
terms of an integral identity~\cite{GurMatZam05}.

Note that generalized solutions a priori belong to the space
$W^m(G)$, whereas Condition~\ref{condSeparK23} is formulated for
functions belonging to the space $W^{2m}$ outside the set
$\mathcal K$. Such a formulation can be justified by the following
result (see Lemma~2.1 in~\cite{GurMatZam05} and Lemma~5.1
in~\cite{GurRJMP04}).

\begin{lemma}\label{lSmoothOutsideK}
Let $u\in W^{m}(G)$ be a generalized solution of
problem~\eqref{eqPinG}, \eqref{eqBinG} with right-hand side $f\in
W^k(G)$. Then
$$
\|u\|_{W^{k+2m}(G\setminus\overline{\mathcal O_\delta(\mathcal
K)})}\le c_{\delta}\left(\|f\|_{W^k(G\setminus\overline{\mathcal
O_{\delta_1}(\mathcal K)})}+\|u\|_{L_2(G)}\right)\quad
\forall\delta>0,
$$
where $\delta_1=\delta_1(\delta)>0$ and $c_{\delta}>0$ do not
depend on $u$.
\end{lemma}

\begin{theorem}[see Theorem~2.1 in~\cite{GurMatZam05}]\label{thTheor2.1GurMatZam04}
Let Conditions~$\ref{condRegLocalProb}$--$\ref{condSeparK23}$
hold. Then the operator $\mathbf P$ has the Fredholm
property.\footnote{See Definition~\ref{defFredholm}.}
\end{theorem}

The aim of this paper is to investigate the influence of
lower-order terms in~\eqref{eqPinG} and nonlocal operators
$\mathbf B_{i\mu}^1$ and $\mathbf B_{i\mu}^2$ in~\eqref{eqBinG}
upon the index of the operator $\mathbf P$.

\subsection{Nonlocal Problems near the Set $\mathcal
K$}\label{subsecSettingNearK}

When studying problem~(\ref{eqPinG}), (\ref{eqBinG}), one must pay
particular attention to the behavior of solutions near the set
${\mathcal K}$ of conjugation points. Let us consider the
corresponding model problems. Denote by $u_j(y)$ the function
$u(y)$ for $y\in{\mathcal O}_{\varepsilon_1}(g_j)$. If
$g_j\in\overline{\Gamma_i},$ $y\in{\mathcal
O}_{\varepsilon}(g_j),$ and $\Omega_{is}(y)\in{\mathcal
O}_{\varepsilon_1}(g_k),$ then we denote the function
$u(\Omega_{is}(y))$ by $u_k(\Omega_{is}(y))$. In this notation,
nonlocal problem~(\ref{eqPinG}), (\ref{eqBinG}) acquires the
following form in the $\varepsilon$-neighborhood of the set
(orbit) $\mathcal K$:
\begin{gather*}
 \mathbf A(y, D_y) u_j=f(y) \quad (y\in\mathcal O_\varepsilon(g_j)\cap
 G),\\
\begin{aligned}
B_{i\mu 0}(y, D_y)u_j(y)|_{\mathcal
O_\varepsilon(g_j)\cap\Gamma_i}+ \sum\limits_{s=1}^{S_i}
\big(B_{i\mu s}(y,D_y)(\zeta
u_k)\big)\big(\Omega_{is}(y)\big)\big|_{\mathcal
O_\varepsilon(g_j)\cap\Gamma_i}
=f_{i\mu}(y) \\
\big(y\in \mathcal O_\varepsilon(g_j)\cap\Gamma_i;\ i\in\{1\le
i\le N: g_j\in\overline{\Gamma_i}\};\ j=1, \dots, N;\ \mu=1,
\dots, m\big),
\end{aligned}
\end{gather*}
where $f_{i\mu}=-\mathbf B_{i\mu}^2u$.

Let $y\mapsto y'(g_j)$ be the change of variables described in
Sec.~\ref{subsecSetting}. Denote $ K_j^\varepsilon=K_j\cap\mathcal
O_\varepsilon(0)$, $
\gamma_{j\sigma}^\varepsilon=\gamma_{j\sigma}\cap\mathcal
O_\varepsilon(0). $ Introduce the functions
$$
U_j(y')=u_j(y(y')),\ f_j(y')=f(y(y')),\ y'\in
K_j^\varepsilon,\qquad f_{j\sigma\mu}(y')=f_{i\mu}(y(y')),\
y'\in\gamma_{j\sigma}^\varepsilon,
$$
where $\sigma=1$ $(\sigma=2)$ if, under the transformation
$y\mapsto y'(g_j)$, the curve $\Gamma_i$ is mapped to the side
$\gamma_{j1}$ ($\gamma_{j2}$) of the angle $K_j$. Denote $y'$ by
$y$ again. Then, by virtue of Condition~\ref{condK1},
problem~(\ref{eqPinG}), (\ref{eqBinG}) acquires the form
\begin{gather}
  \mathbf A_j(y, D_y)U_j=f_{j}(y) \quad (y\in
  K_{j}^\varepsilon),\label{eqPinK}\\
 \sum\limits_{k,s}
       (B_{j\sigma\mu ks}(y, D_y)U_k)({\mathcal G}_{j\sigma ks}y)
    =f_{j\sigma\mu}(y) \quad (y\in\gamma_{j\sigma}^\varepsilon);\label{eqBinK}
\end{gather}
here  $j, k=1, \dots, N;$ $\sigma=1, 2;$ $\mu=1, \dots, m;$ $s=0,
\dots, S_{j\sigma k}$; ${\mathbf A}_j(y, D_y)$ and $B_{j\sigma\mu
ks}(y, D_y)$ are differential operators of order $2m$ and
$m_{j\sigma\mu}$ ($m_{j\sigma\mu}\le m-1$), respectively, with
$C^\infty$ complex-valued coefficients; ${\mathcal G}_{j\sigma
ks}$ is the operator of rotation by an angle~$\omega_{j\sigma ks}$
and the homothety with a coefficient~$\chi_{j\sigma ks}$
($\chi_{j\sigma ks}>0$). Moreover, $ |(-1)^\sigma
b_{j}+\omega_{j\sigma ks}|<b_{k}$ for $(k,s)\ne(j,0) $ (cf.
Remark~\ref{remK1}) and $ \omega_{j\sigma j0}=0$, $\chi_{j\sigma
j0}=1 $ (i.e., ${\mathcal G}_{j\sigma j0}y\equiv y$).

Set $D_\chi=2\max\{\chi_{j\sigma ks}\}.$ The following lemma
establishes the regularity property for solutions of nonlocal
problems near the set $\mathcal K$.
\begin{lemma}[see\footnote{Lemma~2.3 in~\cite{GurMatZam05} was formulated
for $a>2m-1$. However, its proof remains true for any $a\in\mathbb
R$.} Lemma~2.3 in~\cite{GurMatZam05}]\label{lAppL2.3GurMatZam05}
Let $(U_1,\dots,U_N)$ be a solution of problem~\eqref{eqPinK},
\eqref{eqBinK} such that
$$
U_j\in W^{2m}(K_{j}^{D_\chi \varepsilon}\cap\{|y|>\delta\})\quad
\forall \delta>0,\qquad U_j\in H_{a-2m}^0(K_j^{D_\chi
\varepsilon}),
$$
where $a\in\mathbb R$. Suppose that $ f_j\in
H_{a}^{0}(K_j^{\varepsilon})$ and $ f_{j\sigma\mu}\in
H_{a}^{2m-m_{j\sigma\mu}-1/2}(\gamma_{j\sigma}^\varepsilon). $
Then
$$
\sum\limits_j\|U_j\|_{H_a^{2m}(K_j^{\varepsilon/D_\chi^3})}\le
c\sum\limits_j\Big(\|f_j\|_{H_{a}^{0}(K_j^{\varepsilon})}+\sum\limits_{\sigma,\mu}
\|f_{j\sigma\mu}\|_{H_{a}^{2m-m_{j\sigma\mu}-1/2}(\gamma_{j\sigma}^\varepsilon)}+
\|U_j\|_{H_{a-2m}^0(K_{j}^{\varepsilon})}\Big).
$$
\end{lemma}

We write the principal homogeneous parts of the operators $\mathbf
A_j(0, D_y)$ and $B_{j\sigma\mu ks}(0, D_y)$ in the polar
coordinates, $ r^{-2m}\tilde{\mathcal A}_j (\omega, D_\omega,
rD_r),$ $r^{-m_{j\sigma\mu}}\tilde B_{j\sigma\mu ks}(\omega,
D_\omega, rD_r), $ respectively, and consider the analytic
operator-valued function
$$
\tilde{\mathcal L}(\lambda): \prod_{j=1}^{N} W^{l+2m}(-\omega_j,
 \omega_j)\to \prod_{j=1}^{N}\big(W^l(-\omega_j, \omega_j) \times{\mathbb
 C}^{2m}\big),
$$
$$
 \tilde{\mathcal L}(\lambda)\varphi=\big\{\tilde{\mathcal A}_j(\omega, D_\omega, \lambda)\varphi_j,\
  \sum\limits_{k,s} (\chi_{j\sigma ks})^{i\lambda-m_{j\sigma\mu}}
 {\tilde B}_{j\sigma\mu ks}(\omega, D_\omega, \lambda)
              \varphi_k(\omega+\omega_{j\sigma ks})|_{\omega=(-1)^\sigma
              \omega_j}\big\}.
$$
Basic definitions and facts concerning eigenvalues, eigenvectors,
and associate vectors of analytic operator-valued functions can be
found in~\cite{GS}. In the sequel, it will be on principle that
the spectrum of the operator $\tilde{\mathcal L}(\lambda)$ is
discrete (see Lemma~2.1 in~\cite{SkDu90}).

\section{Perturbations by Lower-Order Terms}\label{secPertLOT}

\subsection{Reduction to weighted spaces}
Introduce the lower-order terms operator
\begin{equation}\label{eqOpA'}
A'(y,D_y)=\sum\limits_{|\alpha|\le 2m-1}a_\alpha(y)D^\alpha,
\end{equation}
where $a_\alpha\in C^\infty(\mathbb R^2)$. Consider the perturbed
operator $\mathbf P': \Dom(\mathbf P')\subset L_2(G)\to L_2(G)$
given by
$$
 \mathbf P' u=\mathbf A(y, D_y)u+A'(y,D_y)u,\qquad u\in \Dom(\mathbf P')=\{u\in
 W^m(G,\mathbf B):\ \mathbf A(y, D_y)u+A'(y,D_y)u\in L_2(G)\}.
$$

By Theorem~\ref{thTheor2.1GurMatZam04}, the unbounded operator
$\mathbf P'$ has the Fredholm property (just as $\mathbf P$ has).
The main result of this section (to be proved in
Sec.~\ref{subsecCompLowerTermWeighted}) is as follows.
\begin{theorem}\label{thIndP=IndP'} Let
Conditions~{\rm \ref{condRegLocalProb}--\ref{condSeparK23}} hold.
Then $\ind {\bf P'}=\ind{\bf P}$.
\end{theorem}

This theorem shows that the lower-order terms in~\eqref{eqPinG} do
not affect the index of the unbounded operator $\mathbf P$. The
difficulty is that the above perturbations are, in general,
neither compact nor $\mathbf P$-compact in the sense of
Definition~\ref{defPCompact}. If $m=1$ then $u\in\Dom(\mathbf P)$
implies only $u\in W^1(G)$, which ensures the $\mathbf
P$-boundedness of the perturbation but not its $\mathbf
P$-compactness. However, if $m\ge 2$, then $u\in\Dom(\mathbf P)$
does not imply $u\in W^{2m-1}(G)$, and the perturbation is not
even $\mathbf P$-bounded. Moreover, $\Dom(\mathbf P')\ne
\Dom(\mathbf P)$ in the latter case.

To overcome this difficulty, we introduce the operator $\mathbf Q:
\Dom(\mathbf Q)\subset L_2(G)\to H_a^0(G)$ given by
\begin{equation}\label{eqOpQ}
 \mathbf Q u=\mathbf A(y, D_y)u,\qquad u\in \Dom(\mathbf Q)=\{u\in
 W^m(G,\mathbf B):\ \mathbf A(y, D_y) u\in H_a^0(G)\}.
\end{equation}
In this definition and further (unless otherwise stated), we
assume that
$$
 m-1<a<m.
$$
We will prove that $\ind \mathbf Q=\ind\mathbf P$. On the other
hand, we will show that the operator $A'(y,D_y)$ is a $\mathbf
Q$-compact perturbation, and, therefore, it does not change the
index of $\mathbf Q$ and hence $\mathbf P$.

\begin{lemma}\label{lIndQ=IndP}
Let the line $\Im\lambda=a+1-2m$ contain no eigenvalues of the
operator $\tilde{\mathcal L}(\lambda)$. Then the operator $\mathbf
Q$ has the Fredholm property and $\ind \mathbf Q=\ind\mathbf P$.
\end{lemma}
\begin{proof}
1. It is shown in~\cite[Sec.~6]{GurRJMP04} that $ \mathbf
B_{i\mu}u\in H_a^{2m-m_{i\mu}-1/2}(\Gamma_i)\dotplus
R_a^{i\mu}(\Gamma_i)$ for $u\in H_a^{2m}(G), $ where
$R_a^{i\mu}(\Gamma_i)$ is a finite-dimensional subspace in
$H_{a'}^{2m-m_{i\mu}-1/2}(\Gamma_i)$ for any $a'>2m-1$. Set
$$
\mathcal
H_a^0(G,\Gamma)=H_a^0(G)\times\prod\limits_{i=1}^N\prod\limits_{\mu=1}^m
H_a^{2m-m_{i\mu}-1/2}(\Gamma_i),\qquad \mathcal
R_a^0(G,\Gamma)=\{0\}\times\prod\limits_{i=1}^N\prod\limits_{\mu=1}^m
R_a^{i\mu}(\Gamma_i).
$$
By Theorem~6.1 in~\cite{GurRJMP04}, the bounded operator
\begin{equation}\label{eqOpL}
\mathbf L=\{{\bf A}(y,D_y),{\bf B}_{i\mu}\}: H_a^{2m}(G)\to
\mathcal H_a^0(G,\Gamma)\dotplus \mathcal R_a^0(G,\Gamma)
\end{equation}
has the Fredholm property. Therefore, by virtue of the compactness
of the embedding $H_a^{2m}(G)\subset L_2(G)$ (see
Lemma~\ref{lAppALem3.5Kondr}) and by Theorem~\ref{thAppBLem7.1Kr},
we have
\begin{equation}\label{eqIndQ=IndP_1}
\|u\|_{H_a^{2m}(G)}\le k_1\left(\|\mathbf L u\|_{\mathcal
H_a^0(G,\Gamma)\dotplus \mathcal
R_a^0(G,\Gamma)}+\|u\|_{L_2(G)}\right).
\end{equation}

2. Introduce the unbounded operator $\dot{\mathbf Q}:
\Dom(\dot{\mathbf Q})\subset L_2(G)\to H_a^0(G)$ given by
\begin{equation}\label{eqOpDotQ}
 \dot{\mathbf Q} u=\mathbf A(y, D_y)u,\qquad u\in \Dom(\dot{\mathbf Q})=\{u\in
 H_a^{2m}(G):\ \mathbf B_{i\mu}u=0\}.
\end{equation}
Since $H_a^{2m}(G)\subset W^m(G)$, it follows that $\dot{\mathbf
Q}$ is a restriction of ${\bf Q}$, i.e., $\dot{\mathbf
Q}\subset{\bf Q}.$

First, we prove that $\dot{\mathbf Q}$ has the Fredholm property.
Let $u\in\Dom(\dot{\mathbf Q})$; then $u\in\Dom(\mathbf
L)=H_a^{2m}(G)$ and $\mathbf A(y,D_y)u\in H_a^0(G)$, $\mathbf
B_{i\mu}u=0$. Therefore, estimate~\eqref{eqIndQ=IndP_1} acquires
the form
\begin{equation}\label{eqIndQ=IndP_3}
\|u\|_{H_a^{2m}(G)}\le k_1\left(\|\dot{\mathbf
Q}u\|_{H_a^0(G)}+\|u\|_{L_2(G)}\right)\quad \forall
u\in\Dom(\dot{\mathbf Q}).
\end{equation}

It follows from~\eqref{eqIndQ=IndP_3} that the operator
$\dot{\mathbf Q}$ is closed, $\dim\ker\dot{\mathbf Q}<\infty$, and
$\mathcal R(\dot{\mathbf Q})=\overline{\mathcal R(\dot{\mathbf
Q})}$ (to obtain the latter two properties, one must apply
Theorem~\ref{thAppBLem7.1Kr}).

Let us prove that $\codim\mathcal R(\dot{\mathbf Q})<\infty$.
Since $\mathbf L$ has the Fredholm property, there exist finitely
many linearly independent functions $F_1,\dots,F_d\in H_a^0(G)$
such that a function $f\in H_a^0(G)$ belongs to the image of
$\dot{\mathbf Q}$ iff $(f,F_j)_{H_a^0(G)}=0$, $j=1,\dots, d$.
Thus, $\dot{\mathbf Q}$ has the Fredholm property.

3. Now we prove that $\mathbf Q$ has the Fredholm property. Since
$\ker {\bf Q}=\ker {\bf P}$ and ${\bf P}$ has the Fredholm
property, it follows that
\begin{equation}\label{eqIndQ=IndP_5}
\dim\ker {\bf Q}=\dim\ker {\bf P}<\infty.
\end{equation}
On the other hand, ${\bf Q}$ is an extension of the Fredholm
operator $\dot{\bf Q}$; therefore,
\begin{equation}\label{eqIndQ=IndP_6}
\mathcal R(\mathbf Q)=\overline{\mathcal R(\mathbf Q)},\qquad
\codim\mathcal R(\mathbf Q)<\infty.
\end{equation}

Thus, ${\bf Q}$ is an extension of the Fredholm operator
$\dot{\mathbf Q}$, and possesses properties~\eqref{eqIndQ=IndP_5}
and~\eqref{eqIndQ=IndP_6}. Applying Theorem~\ref{thAppBdotPP}, we
see that ${\bf Q}$ has the Fredholm property.

4. By virtue of \eqref{eqIndQ=IndP_5}, it remains to prove that
$\codim\mathcal R(\mathbf Q)=\codim\mathcal R(\mathbf P)$.

Let $\codim\mathcal R({\bf Q})=d_1$, where $d_1\le d$. Take an
arbitrary function $f\in L_2(G)$. Then $f\in\mathcal R({\bf P})$
iff $f\in\mathcal R({\bf Q})$ because $L_2(G)\subset H_a^0(G)$.
However, the belonging $f\in\mathcal R({\bf Q})$ is equivalent to
the relations $(f,F_j)_{H_a^0(G)}=0$, $j=1,\dots, d_1, $  where
$F_1,\dots,F_{d_1}\in H_a^0(G)$ are linearly independent
functions. Using Schwarz' inequality, the boundedness of the
embedding $L_2(G)\subset H_a^0(G)$, and Riesz' theorem, we see
that these relations are equivalent to the following ones: $
(f,f_j)_{L_2(G)}=0,$ $j=1,\dots, d_1, $ where $f_j\in L_2(G)$.
Moreover, the functions $f_1,\dots,f_{d_1}$ are linearly
independent. (Otherwise, some linear combination of the functions
$F_1,\dots,F_{d_1}$ would be orthogonal in $H_a^{0}(G)$ to any
function from $L_2(G)$. This is impossible because
$F_1,\dots,F_{d_1}$ are linearly independent, while $L_2(G)$ is
dense in $H_a^{0}(G)$.) Thus, we have proved that $\codim\mathcal
R(\mathbf P)=d_1$.
\end{proof}

Introduce the perturbed operator $\mathbf Q': \Dom(\mathbf
Q')\subset L_2(G)\to H_a^0(G)$ given by
$$
 \mathbf Q' u=\mathbf A(y, D_y)u+A'(y,D_y)u,\quad u\in \Dom(\mathbf Q')=\{u\in
 W^m(G,\mathbf B):\ \mathbf A(y, D_y)u+A'(y,D_y)u\in H_a^0(G)\}.
$$
In the following section, we prove that $\ind \mathbf Q'=\ind
\mathbf Q$, provided that the line $\Im\lambda=a+1-2m$ contains no
eigenvalues of the operator $\tilde{\mathcal L}(\lambda)$. Then,
using the discreteness of the spectrum of $\tilde{\mathcal
L}(\lambda)$ and Lemma~\ref{lIndQ=IndP}, we will prove
Theorem~\ref{thIndP=IndP'}.

\subsection{Compactness of lower-order terms in weighted
spaces}\label{subsecCompLowerTermWeighted}

\begin{lemma}\label{lAprEstWm}
Let the line $\Im\lambda=a+1-2m$ contain no eigenvalues of the
operator $\tilde{\mathcal L}(\lambda)$. Then
$$
\|u\|_{W^m(G)}\le c\left(\|{\mathbf
Q}u\|_{H_a^0(G)}+\|u\|_{L_2(G)}\right)\quad \forall u\in\Dom({\bf
Q}).
$$
\end{lemma}
\begin{proof}
Consider the unbounded operator $\hat{\mathbf Q}:
\Dom(\hat{\mathbf Q})\subset W^m(G)\to H_a^0(G)$ given by $
\hat{\mathbf Q} u=\mathbf A(y, D_y)u,$ $u\in \Dom(\hat{\mathbf
Q})=\Dom(\mathbf Q). $ Since ${\mathbf Q}$ has the Fredholm
property, the same is true for $\hat{\mathbf Q}$. Therefore, the
desired estimate follows from the compactness of the embedding
$W^m(G)\subset L_2(G)$ and from Theorem~\ref{thAppBLem7.1Kr}.
\end{proof}

Take a number $b$ such that
\begin{equation}\label{eqb}
m-1<b<a<m.
\end{equation}

Consider a function $\psi_j\in C_0^\infty(\mathbb R^2)$ equal to
$1$ in a small neighborhood of the point $g_j\in\mathcal K$ and
vanishing outside a larger neighborhood of $g_j$. The following
lemma describes the behavior of $u\in\Dom({\bf Q})$ near the set
$\mathcal K$.
\begin{lemma}\label{lAsymp}
For any $u\in\Dom({\bf Q})$, we have
\begin{equation}\label{eqAsymp_1}
u(y)=\sum\limits_{j=1}^N P_j(y)+v(y),
\end{equation}
where
\begin{equation}\label{eqAsymp_1'}
P_j(y)=\psi_j(y)\sum\limits_{|\alpha|\le m-2}p_{j\alpha}
(y-g_j)^\alpha,\qquad p_{j\alpha}\in\mathbb C,
\end{equation}
 and $v\in H_{b+1}^{2m}(G)$ {\rm (}if $m=1$, we set $P_j(y)\equiv0${\rm );} moreover,
\begin{equation}\label{eqAsymp_2}
\sum\limits_{j,\alpha}|p_{j\alpha}|+\|v\|_{H_{b+1}^{2m}(G)}\le
c\left(\|{\mathbf Q}u\|_{H_a^0(G)}+\|u\|_{L_2(G)}\right).
\end{equation}
\end{lemma}
\begin{proof}
1. It follows from Lemma~\ref{lSmoothOutsideK} that $u\in
W^{2m}(G\setminus\overline{\mathcal O_\delta(\mathcal K)})$ for
any $\delta>0$ and
\begin{equation}\label{eqAsymp_3}
\|u\|_{W^{2m}(G\setminus\overline{\mathcal O_\delta(\mathcal
K)})}\le k_{1\delta}\left(\|{\mathbf
Q}u\|_{H_a^0(G)}+\|u\|_{L_2(G)}\right),
\end{equation}
where $k_{1\delta}$ does not depend on $u$. Therefore, it suffices
to consider the behavior of $u$ near the set $\mathcal K$.

By Lemma~\ref{lu=P+v}, $u\in W^m(G)$ can be represented in the
form~\eqref{eqAsymp_1}, where $P_j(y)$ is given
by~\eqref{eqAsymp_1'}, $v\in H_{b-m+1}^m(G)$, and
\begin{equation}\label{eqAsymp_4}
\sum\limits_{j,\alpha}|p_{j\alpha}|+\|v\|_{H_{b-m+1}^m(G)}\le
k_2\left(\|{\mathbf Q}u\|_{H_a^0(G)}+\|u\|_{L_2(G)}\right)
\end{equation}
(to obtain~\eqref{eqAsymp_4}, we have also applied
Lemma~\ref{lAprEstWm}).

Moreover, relations~\eqref{eqAsymp_1}, \eqref{eqAsymp_3},
and~\eqref{eqAsymp_4} imply that
\begin{equation}\label{eqAsymp_4'}
\|v\|_{W^{2m}(G\setminus\overline{\mathcal O_\delta(\mathcal
K)})}\le k_{2\delta}\left(\|{\mathbf
Q}u\|_{H_a^0(G)}+\|u\|_{L_2(G)}\right)\quad \forall \delta>0,
\end{equation}
where $k_{2\delta}$ does not depend on $u$. It remains to prove
that $v\in H_{b+1}^{2m}(G)$.

2. By using~(\ref{eqPinG}) and (\ref{eqBinG}), we see that $v$ is
a solution of the problem
\begin{equation}\label{eqAsymp_5}
 {\bf A}(y, D_y)v=f-{\bf A}(y, D_y)P\equiv f',\qquad
     \mathbf B_{i\mu}^0 v+ \mathbf B_{i\mu}^1 v=
     -\mathbf B_{i\mu}P-\mathbf B_{i\mu}^2 v\equiv f_{i\mu}',
\end{equation}
where $ P(y)=\sum_{j=1}^N P_j(y) $ and $f=\mathbf Q u\in
H_a^0(G)$. It follows from the boundedness of the embedding
$H_a^0(G)\subset H_{b+1}^0(G)$ (see~\eqref{eqb}) and from the
estimate of the coefficients $p_{j\alpha}$ (see~\eqref{eqAsymp_4})
that
\begin{equation}\label{eqAsymp_6}
\|f'\|_{H_{b+1}^0(G)}\le k_3\left(\|{\mathbf
Q}u\|_{H_a^0(G)}+\|u\|_{L_2(G)}\right).
\end{equation}
Similarly, using additionally inequalities~\eqref{eqSeparK23'}
and~\eqref{eqAsymp_4'}, we obtain $f_{i\mu}'=-\mathbf
B_{i\mu}P-\mathbf B_{i\mu}^2 v\in W^{2m-m_{i\mu}-1/2}(\Gamma_i)$
and
\begin{equation}\label{eqAsymp_7}
\|f_{i\mu}'\|_{W^{2m-m_{i\mu}-1/2}(\Gamma_i)}\le
k_4\left(\|{\mathbf Q}u\|_{H_a^0(G)}+\|u\|_{L_2(G)}\right).
\end{equation}

On the other hand, $v\in H_{b-m+1}^m(G)$; hence $f_{i\mu}'=\mathbf
B_{i\mu}^0 v+ \mathbf B_{i\mu}^1 v\in
H_{b-m+1}^{m-m_{i\mu}-1/2}(\Gamma_i)$. We claim that
\begin{equation}\label{eqAsymp_8}
\|f_{i\mu}'\|_{H_{b+1}^{2m-m_{i\mu}-1/2}(\Gamma_i)}\le
k_5\left(\|{\mathbf Q}u\|_{H_a^0(G)}+\|u\|_{L_2(G)}\right).
\end{equation}
To prove this assertion, we fix $i$ and $\mu$ and set
$\Gamma=\Gamma_i$. Let $g\in\overline{\Gamma}\setminus\Gamma$.
Assume, without loss of generality, that $g=0$ and $\Gamma$
coincides with the axis $Oy_1$ in a sufficiently small
neighborhood $\mathcal O_\varepsilon(0)$ of the origin. Denote
$$
G^\varepsilon=G\cap\mathcal O_\varepsilon(0),\qquad
\Gamma^\varepsilon=\Gamma\cap\mathcal O_\varepsilon(0),
$$
in which case $H_a^k(G^\varepsilon)=H_a^k(G^\varepsilon,\{0\})$.

Using part~1 of Lemma~\ref{lPsi=P1+Phi}, we represent
$f_{i\mu}'\in W^{2m-m_{i\mu}-1/2}(\Gamma^\varepsilon)$ near the
origin as follows:
$$
f_{i\mu}'(r)=P_1(r)+f_{i\mu}''(r),\qquad 0<r<\varepsilon,
$$
where $ P_1(r)$ is a polynomial of order $2m-m_{i\mu}-2$, whereas
$f_{i\mu}''\in H_{b+1}^{2m-m_{i\mu}-1/2}(\Gamma^\varepsilon)$ (in
fact, we can replace $b+1$ by any positive number in the last
relation). Now we have $f_{i\mu}',f_{i\mu}''\in
H_{b-m+1}^{m-m_{i\mu}-1/2}(\Gamma^\varepsilon)$; therefore,
$P_1\in H_{b-m+1}^{m-m_{i\mu}-1/2}(\Gamma^\varepsilon)$, i.e.,
$P_1$ consists of monomials of order greater than or equal to
$m-m_{i\mu}-1$. This implies that $P_1\in
H_{b+1}^{2m-m_{i\mu}-1/2}(\Gamma^\varepsilon)$. Using part~3 of
Lemma~\ref{lPsi=P1+Phi}, we obtain
$$
\|f_{i\mu}'\|_{H_{b+1}^{2m-m_{i\mu}-1/2}(\Gamma^\varepsilon)}\le
\|P_1\|_{H_{b+1}^{2m-m_{i\mu}-1/2}(\Gamma^\varepsilon)}+
\|f_{i\mu}''\|_{H_{b+1}^{2m-m_{i\mu}-1/2}(\Gamma^\varepsilon)}\le
k_6\|f_{i\mu}'\|_{W^{2m-m_{i\mu}-1/2}(\Gamma^\varepsilon)}.
$$
Combining this estimate with~\eqref{eqAsymp_7}
yields~\eqref{eqAsymp_8}.

3. Applying Lemma~\ref{lAppL2.3GurMatZam05} to
problem~\eqref{eqAsymp_5} and taking into
account~\eqref{eqAsymp_4'}, \eqref{eqAsymp_6}, \eqref{eqAsymp_8},
and~\eqref{eqAsymp_4}, we obtain
$$
\|v\|_{H_{b+1}^{2m}(G)}\le k_7\left(\|f'\|_{H_{b+1}^0(G)}+
\sum\limits_{i,\mu}\|f_{i\mu}'\|_{H_{b+1}^{2m-m_{i\mu}-1/2}(\Gamma_i)}+
\|v\|_{H_{b-2m+1}^{0}(G)}\right) \le k_8 \left(\|{\mathbf
Q}u\|_{H_a^0(G)}+\|u\|_{L_2(G)}\right).
$$
Combining this inequality with~\eqref{eqAsymp_4}
yields~\eqref{eqAsymp_2}.
\end{proof}

The following corollary results from Lemma~\ref{lAsymp}.
\begin{corollary}\label{corA'u}
Let $A'(y,D_y)$ be the differential operator of order $2m-1$,
given by~\eqref{eqOpA'}. Then
\begin{equation}\label{eqAprEstWm_0}
\|A'(y,D_y)u\|_{H_{b+1}^1(G)}\le c\left(\|{\bf
Q}u\|_{H_a^0(G)}+\|u\|_{L_2(G)}\right)\quad \forall u\in\Dom({\bf
Q}).
\end{equation}
\end{corollary}

Now we can prove that lower-order perturbations in~\eqref{eqPinG}
do not change the index of $\mathbf Q$.

\begin{lemma}\label{lIndQ=IndQ'}
Let the line $\Im\lambda=a+1-2m$ contain no eigenvalues of the
operator $\tilde{\mathcal L}(\lambda)$. Then the operators
$\mathbf Q$ and $\mathbf Q'$ have the Fredholm property and
$\ind\mathbf Q'=\ind\mathbf Q$.
\end{lemma}
\begin{proof}
By Lemma~\ref{lIndQ=IndP}, $\mathbf Q$ and $\mathbf Q'$ have the
Fredholm property.

Introduce the operator $\mathbf A': \Dom(\mathbf A')\subset
L_2(G)\to H_a^0(G)$ given by $
 \mathbf A' u=\mathbf A'(y, D_y)u$, $u\in \Dom(\mathbf A')=\Dom({\mathbf Q}).$
It follows from Corollary~\ref{corA'u} and from the compactness of
the embedding $H_{b+1}^1(G)\subset H_a^0(G)$ (see~\eqref{eqb} and
Lemma~\ref{lAppALem3.5Kondr}) that $\mathbf Q'=\mathbf Q+\mathbf
A'$ and $\mathbf A'$ is a $\mathbf Q$-compact operator. Therefore,
by Theorem~\ref{thAppBTheor5.26Kato}, we have $\ind\mathbf
Q'=\ind\mathbf Q$.
\end{proof}

\begin{proof}[Proof of Theorem~$\ref{thIndP=IndP'}$]
It follows from Lemma~2.1 in~\cite{SkDu90} that the spectrum of
$\tilde{\mathcal L}(\lambda)$ is discrete. Therefore, one can find
a number $a$ such that $m-1<a<m$ and the line $\Im\lambda=a+1-2m$
contains no eigenvalues of $\tilde{\mathcal L}(\lambda)$. In this
case, Lemmas~\ref{lIndQ=IndP} and~\ref{lIndQ=IndQ'} imply $
\ind\mathbf P'=\ind\mathbf Q'=\ind\mathbf Q=\ind\mathbf P. $
\end{proof}

\section{Perturbations in Nonlocal Conditions}\label{secPertNC}

\subsection{Formulation of the main result}

In this section, we investigate the stability of index for
nonlocal operators under the perturbation of nonlocal conditions
by operators which have the same form as $\mathbf B_{i\mu}^1$ and
$\mathbf B_{i\mu}^2$. This situation is more difficult than that
in Sec.~\ref{secPertLOT} because the above perturbations
explicitly change the domain of the corresponding unbounded
operators. Therefore, these perturbations cannot be treated as
relatively compact ones, and we make use of another approach based
on the notion of the {\it gap between closed operators}.

We consider differential operators $C_{i\mu s}(y,D_y)$, $i=1,
\dots, N$, $\mu=1, \dots, m$, $s=1, \dots, S_i'$, of the same
order $m_{i\mu}$ as $B_{i\mu s}$ in Sec.~\ref{subsecSetting},
given by
$$
C_{i\mu s}(y,D_y)u=\sum\limits_{|\alpha|\le m_{i\mu}}c_{i\mu
s\alpha}(y)D^\alpha u,
$$
where $c_{i\mu s\alpha}\in C^\infty(\mathbb R^2)$. Introduce the
operator $\mathbf C_{i\mu}^1$ by the formula
$$
\mathbf C_{i\mu}^1u=\sum\limits_{s=1}^{S_i'}\big(C_{i\mu s}(y,
   D_y)(\zeta u)\big)\big(\Omega_{is}'(y)\big),\
   y\in\Gamma_i\cap\mathcal O_{\varepsilon}(\mathcal K),\qquad
 \mathbf C_{i\mu}^1u=0,\
y\in\Gamma_i\setminus\mathcal O_{\varepsilon}(\mathcal K),
$$
where $\zeta$ and $\varepsilon$ are the same as in the definition
of $\mathbf B_{i\mu}^1$, whereas $\Omega_{is}'$ are
$C^\infty$-diffeomorphisms possessing the same properties as
$\Omega_{is}$ (in particular, they satisfy Condition~\ref{condK1}
with $S_i$ and $\Omega_{is}$ replaced by $S_i'$ and
$\Omega_{is}'$).

We also consider operators $\mathbf C_{i\mu}^2$ satisfying
Condition~\ref{condSeparK23} with $\mathbf B_{i\mu}^2$ replaced by
$\mathbf C_{i\mu}^2$. Set
$$
\mathbf C_{i\mu}=\mathbf C_{i\mu}^1+\mathbf C_{i\mu}^2.
$$

We prove an index stability theorem under the following conditions
(which are assumed to hold along with
Conditions~\ref{condRegLocalProb}--\ref{condSeparK23} throughout
this sections, including the formulation of lemmas).

\begin{condition}[see, e.g.,~\cite{LM}]\label{condNormal} The system $\{\mathbf
B_{i\mu}^0\}_{\mu=1}^m$ is normal on $\overline{\Gamma_i}$, $i=1,
\dots, N$.
\end{condition}

\begin{condition}\label{condC1}
$ D^\sigma c_{i\mu s\alpha}(g_{i1})=D^\sigma c_{i\mu
s\alpha}(g_{i2})=0,\quad |\sigma|=0,\dots,
(m-1)-(m_{i\mu}-|\alpha|)$.
\end{condition}

Denote by $g_{i1}$ and $g_{i2}$ the end points of
$\overline{\Gamma_i}$. Let $\tau_{i1}$ ($\tau_{i2}$) be a unit
vector tangent to $\overline{\Gamma_i}$ at the point $g_{i1}$
($g_{i2}$).

\begin{condition}\label{condC2} $
\dfrac{\partial^\beta \mathbf
C_{i\mu}^2u}{\partial\tau_{i1}^\beta}\bigg|_{y=g_{i1}}=
\dfrac{\partial^\beta \mathbf
C_{i\mu}^2u}{\partial\tau_{i2}^\beta}\bigg|_{y=g_{i2}}=0,\quad
\beta=0,\dots, m-1-m_{i\mu}$, $\forall u\in
W^{2m}(G\setminus\overline{\mathcal O_{\varkappa_1}(\mathcal
  K)})$.
\end{condition}

The following lemma is a consequence of Conditions~\ref{condC1}
and~\ref{condC2} (recall that $m-1<a<m$ throughout).
\begin{lemma}\label{lCinHa}
The following inequalities hold{\rm :}
\begin{equation}\label{lCinHa_1}
\|\mathbf C_{i\mu}^1u\|_{H_a^{2m-m_{i\mu}-1/2}(\Gamma_i)}\le
c_1\|u\|_{H_{a+m}^{2m}(G)},
\end{equation}
\begin{equation}\label{lCinHa_2}
\|\mathbf C_{i\mu}^2u\|_{H_a^{2m-m_{i\mu}-1/2}(\Gamma_i)}\le
c_2\|u\|_{W^{2m}(G\setminus\overline{\mathcal
O_{\varkappa_1}(\mathcal
  K)})}.
\end{equation}
\end{lemma}
\begin{proof}
1. For any $u\in H_{a+m}^{2m}(G)$, we have $ (D^\alpha
u)\big(\Omega_{is}'(y)\big)\big|_{\Gamma_i}\in
H_{a+m}^{2m-|\alpha|-1/2}(\Gamma_i)\subset
H_{a+m-(m_{i\mu}-|\alpha|)}^{2m-m_{i\mu}-1/2}(\Gamma_i). $
Therefore, by Condition~\ref{condC1} and Lemma~\ref{lbPsi}, we
have $ (c_{i\mu s\alpha}D^\alpha
u)\big(\Omega_{is}'(y)\big)\big|_{\Gamma_i}\in
H_a^{2m-m_{i\mu}-1/2}(\Gamma_i). $ Estimate~\eqref{lCinHa_1}
follows from the boundedness of the above embedding and from
inequality~\eqref{eqbPsi_A}.

2. It follows from Condition~\ref{condSeparK23} (applied to
$\mathbf C_{i\mu}^2$) that $\mathbf C_{i\mu}^2u\in
  W^{2m-m_{i\mu}-1/2}(\Gamma_i)$ for any $u\in
W^{2m}(G\setminus\overline{\mathcal O_{\varkappa_1}(\mathcal
  K)})$. Now it follows from
Condition~\ref{condC2} and from Lemma~\ref{lPsig=0} that $ \mathbf
C_{i\mu}^2u\in H_a^{2m-m_{i\mu}-1/2}(\Gamma_i). $
Estimate~\eqref{lCinHa_2} follows from
inequality~\eqref{eqSeparK23'} (applied to $\mathbf C_{i\mu}^2$)
and from~\eqref{eqlPsig=0_2}.
\end{proof}

In this section, we write $\mathbf A=\mathbf A(y,D_y).$ Consider
the operators $\mathbf P_t: \Dom(\mathbf P_t)\subset L_2(G)\to
L_2(G)$, $t\in\mathbb C$, given by
$$
 \mathbf P_t u=\mathbf Au,\qquad u\in \Dom(\mathbf P_t)=\{u\in
 W^m(G,\mathbf B+t\mathbf C):\ \mathbf A u\in L_2(G)\},
$$
where $W^m(G,\mathbf B+t\mathbf C)$ is the space of functions
$u\in W^m(G)$ that satisfy the nonlocal conditions $ (\mathbf
B_{i\mu}^0+\mathbf B_{i\mu}^1+t\mathbf C_{i\mu})u=0. $ The main
result of this section (to be proved in Sec.~\ref{subsecGap}) is
as follows.
\begin{theorem}\label{thIndP0=IndP1}
Let Conditions~{\rm \ref{condRegLocalProb}--\ref{condSeparK23}}
and~{\rm \ref{condNormal}--\ref{condC2}} hold. Then $\ind {\bf
P}_t=\const$  $\forall t\in\mathbb C$.
\end{theorem}

\subsection{The gap between nonlocal operators in weighted
spaces}\label{subsecGap}

As in Sec.~\ref{secPertLOT}, we preliminarily study the operators
$\mathbf Q_t: \Dom(\mathbf Q_t)\subset L_2(G)\to H_a^0(G)$ given
by
$$
 \mathbf Q_t u=\mathbf Au,\qquad u\in \Dom(\mathbf Q_t)=\{u\in
 W^m(G,\mathbf B+t\mathbf C):\ \mathbf Au\in H_a^0(G)\},
$$
where $t\in\mathbb C$ and $W^m(G,\mathbf B+t\mathbf C)$ is the
same as in the definition of the operator $\mathbf P_t$. The
operators $\mathbf P_t$ and $\mathbf Q_t$ correspond to the
problem
\begin{gather}
 \mathbf Au=f(y) \quad (y\in G),\label{eqPinGPert}\\
     (\mathbf B_{i\mu}^0+\mathbf B_{i\mu}^1+t\mathbf C_{i\mu})u=0\quad
   (y\in \Gamma_i;\ i=1, \dots, N;\ \mu=1, \dots, m).\label{eqBinGPert}
\end{gather}

\begin{remark}\label{remOneAndTheSameL}
The operator $\tilde{\mathcal L}(\lambda)$ was constructed in
Sec.~\ref{subsecSettingNearK} by means of principal homogeneous
parts of the operators $\mathbf A$ and $B_{i\mu s}(y, D_y)$ at the
points of the set $\mathcal K$. Due to Condition~\ref{condC1}, the
principal homogeneous parts of the operators $C_{i\mu s}(y, D_y)$
are equal to zero. Therefore, one and the same operator
$\tilde{\mathcal L}(\lambda)$ corresponds to
problem~(\ref{eqPinGPert}), (\ref{eqBinGPert}) for any~$t$.
\end{remark}

Fix a number $a$ such that $m-1<a<m$ and the line
$\Im\lambda=a+1-2m$ contains no eigenvalues of $\tilde{\mathcal
L}(\lambda)$ (which is possible due to the discreteness of the
spectrum of $\tilde{\mathcal L}(\lambda)$). It follows from
Remark~\ref{remOneAndTheSameL} and from Lemma~\ref{lIndQ=IndP}
that $\mathbf Q_t$ has the Fredholm property. Therefore, its graph
$\Gr \mathbf Q_t$ is a closed subspace in the Hilbert space
$L_2(G)\times H_a^0(G)$; this space is endowed with the norm
$$
\|(u,f)\|=\left(\|u\|_{L_2(G)}^2+\|f\|_{H_a^0(G)}^2\right)^{1/2}\quad\forall
(u,f)\in L_2(G)\times H_a^0(G).
$$

Denote
\begin{equation}\label{eqDeltaTT+S}
\delta(\mathbf Q_t,\mathbf Q_{t+s})=\sup\limits_{u\in\Dom(\mathbf
Q_t):\ \|(u, \mathbf Q_tu)\|=1}\dist\big((u, \mathbf Q_tu),\Gr
\mathbf Q_{t+s}\big).
\end{equation}
By Definition~\ref{defGap}, the number $\hat\delta(\mathbf
Q_t,\mathbf Q_{t+s})=\max\{\delta(\mathbf Q_t,\mathbf
Q_{t+s}),\delta(\mathbf Q_{t+s},\mathbf Q_t)\}$ is the {\it gap
between the operators $\mathbf Q_t$ and $\mathbf Q_{t+s}$}.

The main tool which enables us to prove the index stability
theorem is Theorem~\ref{thAppBTheor5.17Kato} and the following
result (to be proved later on).
\begin{theorem}\label{thSmallGap}
Let Conditions~{\rm \ref{condRegLocalProb}--\ref{condSeparK23}}
and~{\rm \ref{condNormal}--\ref{condC2}} hold. Suppose that the
lines $\Im\lambda=a+1-2m$ and $\Im\lambda=a+1-m$ contain no
eigenvalues of $\tilde{\mathcal L}(\lambda)$. Then
\begin{equation}\label{eqSmallGap_1}
\hat\delta(\mathbf Q_t,\mathbf Q_{t+s})\le c_t s,\quad |s|\le s_t,
\end{equation}
where $s_t>0$ is sufficiently small, while $c_t>0$ does not depend
on $s$.
\end{theorem}

First, we prove several auxiliary results.

\begin{lemma}\label{lAprEstQt+s}
Let the line $\Im\lambda=a+1-m$ contain no eigenvalues of
$\tilde{\mathcal L}(\lambda)$. Then
\begin{equation}\label{eqAprEstQt+s}
\|u\|_{H_{a+m}^{2m}(G)}\le c_t\|(u,\mathbf Au)\|\quad \forall
u\in\Dom(\mathbf Q_{t+s}),
\end{equation}
where $c_t>0$ does not depend on $s$ and $u$, provided that $|s|$
is sufficiently small.
\end{lemma}
\begin{proof}
1. Consider the bounded operator
\begin{equation}\label{eqAprEstQt+s_Mt}
\mathbf M_t=\{{\bf A},\mathbf B_{i\mu}^0+\mathbf
B_{i\mu}^1+t\mathbf C_{i\mu}\}: H_{a+m}^{2m}(G)\to \mathcal
H_{a+m}^0(G,\Gamma).
\end{equation}
Since the belonging $v\in H_{a+m}^{2m}(G)$ implies $ (\mathbf
B_{i\mu}^0+\mathbf B_{i\mu}^1+t\mathbf C_{i\mu}^1)v\in
H_{a+m}^{2m-m_{i\mu}-1/2}(\Gamma_i)$ and $\mathbf C_{i\mu}^2v\in
W^{2m-m_{i\mu}-1/2}(\Gamma_i)\subset
H_{a+m}^{2m-m_{i\mu}-1/2}(\Gamma_i) $ (the latter relations are
due to Condition~\ref{condSeparK23} and part~1 of
Lemma~\ref{lPsi=P1+Phi}), it follows that the operator $\mathbf
M_t$ is well defined.

By Theorem~6.1 in~\cite{GurRJMP04} and by
Remark~\ref{remOneAndTheSameL}, the operator $\mathbf M_t$ has the
Fredholm property for any $t\in\mathbb C$. Therefore, applying
Theorem~\ref{thAppBLem7.1Kr} and noting that the embedding
$H_{a+m}^{2m}\subset L_2(G)$ is compact for $a<m$ (see
Lemma~\ref{lAppALem3.5Kondr}), we obtain
\begin{equation}\label{eqAprEstQt+s_1}
\|u\|_{H_{a+m}^{2m}(G)}\le k_1\left(\|\mathbf M_t u\|_{\mathcal
H_{a+m}^0(G,\Gamma)}+\|u\|_{L_2(G)}\right)\quad \forall u\in
H_{a+m}^{2m}(G),
\end{equation}
where $k_1>0$ may depend on $t$ but does not depend on $s$ and
$u$.

2. Now take a function $u\in\Dom(\mathbf Q_{t+s})$.  By
Lemma~\ref{lAsymp}, $u\in H_{a+m}^{2m}(G)$.
Inequality~\eqref{eqAprEstQt+s_1}, estimate~\eqref{eqSeparK23'}
(for $\mathbf C^2_{i\mu}$), and the boundedness of the embedding
$W^{2m-m_{i\mu}-1/2}(\Gamma_i)\subset
H_{a+m}^{2m-m_{i\mu}-1/2}(\Gamma_i)$ (see part~1 of
Lemma~\ref{lPsi=P1+Phi}) yield
$$
\|u\|_{H_{a+m}^{2m}(G)}\le k_1\left(\|\mathbf A
u\|_{H_{a+m}^0(G)}+\|u\|_{L_2(G)}\right)+k_2|s|\cdot\|u\|_{H_{a+m}^{2m}(G)}\quad
\forall u\in\Dom(\mathbf Q_{t+s}),
$$
where $k_2>0$ may depend on $t$ but does not depend on $s$ and
$u$. Choosing $|s|\le 1/(2k_2)$ and noting that the embedding
$H_{a}^0(G)\subset H_{a+m}^0(G)$ is bounded, we
obtain~\eqref{eqAprEstQt+s}.
\end{proof}

Lemmas~\ref{lCinHa} and~\ref{lAprEstQt+s} imply:
\begin{corollary}\label{corAprEstCimu}
Let the line $\Im\lambda=a+1-m$ contain no eigenvalues of
$\tilde{\mathcal L}(\lambda)$. Then
\begin{equation}\label{eqAprEstCimu}
\|\mathbf C_{i\mu}u\|_{H_a^{2m-m_{i\mu}-1/2}(\Gamma_i)}\le
c_t\|(u,\mathbf Au)\|\quad \forall u\in\Dom(\mathbf Q_{t+s}),
\end{equation}
where $c_t>0$ does not depend on $s$ and $u$, provided that $|s|$
is sufficiently small.
\end{corollary}

The following two lemmas enable us to reduce nonlocal problems
with nonhomogeneous nonlocal conditions to nonlocal problems with
homogeneous ones. This is the place where
Condition~\ref{condNormal} is needed.

\begin{lemma}[see Lemma~8.1 in~\cite{GurRJMP04}]\label{lAppL8.1GurRJMP04}
Let $a\in\mathbb R$. For any right-hand sides $f_{j\sigma\mu}\in
H_a^{2m-m_{j\sigma\mu}-1/2}(\gamma_{j\sigma})$ in~\eqref{eqBinK}
such that $\supp f_{j\sigma\mu}\subset
\gamma_{j\sigma}^{\varepsilon/2}$, there exist functions $U_j\in
H_a^{2m}(K_j)$ such that $\supp U_{j}\subset
\overline{K_j^{\varepsilon}}$,
$$
B_{j\sigma\mu j0}(y, D_y)U_j(y)=f_{j\sigma\mu}(y),\quad
(B_{j\sigma\mu ks}(y, D_y)U_k)({\mathcal G}_{j\sigma ks}y)=0,\
y\in\gamma_{j\sigma},\ (k,s)\ne(j,0),
$$
$$
\sum\limits_j\|U_j\|_{H_a^{2m}(K_j)}\le
c\sum\limits_{j,\sigma,\mu}
\|f_{j\sigma\mu}\|_{H_a^{2m-m_{j\sigma\mu}-1/2}(\gamma_{j\sigma})}.
$$
\end{lemma}

\begin{lemma}\label{lHomogeneous}
Let $f_{i\mu}\in H_a^{2m-m_{i\mu}-1/2}(\Gamma_i)$. Then, for
$t\in\mathbb C$ and $|s|\le 1$, there is a function $u\in
H_a^{2m}(G)$ such that
\begin{equation}\label{eqHomogeneous_1}
 (\mathbf B_{i\mu}^0+\mathbf B_{i\mu}^1+(t+s)\mathbf
 C_{i\mu})u=f_{i\mu},
\end{equation}
\begin{equation}\label{eqHomogeneous_2}
\|u\|_{H_a^{2m}(G)}\le
c_t\sum\limits_{i,\mu}\|f_{i\mu}\|_{H_a^{2m-m_{i\mu}-1/2}(\Gamma_i)},
\end{equation}
where $c_t>0$ does not depend on $f_{i\mu}$ and $s$.
\end{lemma}
\begin{proof}
Using Lemma~\ref{lAppL8.1GurRJMP04} and a partition of unity, we
construct a function $v\in H_a^{2m}(G)$ such that
\begin{equation}\label{eqHomogeneous_A}
\supp v\subset \overline{G}\setminus \overline{G_\rho},
\end{equation}
\begin{equation}\label{eqHomogeneous_B}
\mathbf B_{i\mu}^0v=f_{i\mu},\qquad \mathbf B_{i\mu}^1v=0,\qquad
\mathbf C_{i\mu}^1v=0,
\end{equation}
\begin{equation}\label{eqHomogeneous_C}
\|v\|_{H_a^{2m}(G)}\le
k_1\sum\limits_{i,\mu}\|f_{i\mu}\|_{H_a^{2m-m_{i\mu}-1/2}(\Gamma_i)},
\end{equation}
where $k_1>0$ does not depend on $f_{i\mu}$, $t$, and $s$.

By~\eqref{eqHomogeneous_A} and~\eqref{eqSeparK23''}, we have
$\supp \mathbf C_{i\mu}^2v\subset\mathcal O_{\varkappa_2}(\mathcal
K)$. Moreover, by Lemma~\ref{lCinHa}, $\mathbf C_{i\mu}^2v\in
H_a^{2m-m_{i\mu}-1/2}(\Gamma_i)$. Therefore, using
Lemma~\ref{lAppL8.1GurRJMP04} and a partition of unity again, we
construct a function $w\in H_a^{2m}(G)$ such that
\begin{equation}\label{eqHomogeneous_D}
\supp w\subset \mathcal O_{\varkappa_1}(\mathcal K),
\end{equation}
\begin{equation}\label{eqHomogeneous_E}
\mathbf B_{i\mu}^0w=-(t+s)\mathbf C_{i\mu}^2v,\qquad \mathbf
B_{i\mu}^1w=0,\qquad \mathbf C_{i\mu}^1w=0,
\end{equation}
$$
\|w\|_{H_a^{2m}(G)} \le k_1\sum\limits_{i,\mu}\|(t+s)\mathbf
C_{i\mu}^2v\|_{H_a^{2m-m_{i\mu}-1/2}(\Gamma_i)}.
$$
Using the relation $|s|\le1$ and inequalities~\eqref{lCinHa_2}
and~\eqref{eqHomogeneous_C}, we infer from the last inequality
\begin{equation}\label{eqHomogeneous_F}
\|w\|_{H_a^{2m}(G)} \le k_1\sum\limits_{i,\mu}(|t|+1)\|\mathbf
C_{i\mu}^2v\|_{H_a^{2m-m_{i\mu}-1/2}(\Gamma_i)}\le
k_2\|v\|_{H_a^{2m}(G)}\le
k_2k_1\sum\limits_{i,\mu}\|f_{i\mu}\|_{H_a^{2m-m_{i\mu}-1/2}(\Gamma_i)},
\end{equation}
where $k_2>0$ may depend on $t$ but does not depend on $f_{i\mu}$
and $s$.

By~\eqref{eqHomogeneous_D} and~\eqref{lCinHa_2}, we have $\mathbf
C_{i\mu}^2w=0$. It follows from this relation,
from~\eqref{eqHomogeneous_B}, and from~\eqref{eqHomogeneous_E}
that $u=v+w$ satisfies~\eqref{eqHomogeneous_1}.
Inequality~\eqref{eqHomogeneous_2} follows from
inequalities~\eqref{eqHomogeneous_C} and~\eqref{eqHomogeneous_F}.
\end{proof}

\begin{remark}\label{remHomogeneous}
One can easily see that, if $(\mathbf C_{i\mu}^2v)(y)=0$ in
$\mathcal O_\varkappa(\mathcal K)$ for some $\varkappa>0$ and for
any $v\in W^{2m}(G\setminus\overline{\mathcal
O_{\varkappa_1}(\mathcal K)})$, then Lemma~\ref{lHomogeneous} is
true for any $a\in\mathbb R$.
\end{remark}

\begin{proof}[Proof of Theorem~{\rm \ref{thSmallGap}}]
1. We have to prove inequality~\eqref{eqSmallGap_1} for the
quantity $\hat \delta(\mathbf Q_t,\mathbf Q_{t+s})$ replaced by
$\delta(\mathbf Q_t,\mathbf Q_{t+s})$ and $\delta(\mathbf
Q_{t+s},\mathbf Q_{t})$. Let us prove the inequality
\begin{equation}\label{lSmallGapTT+S22}
\delta(\mathbf Q_t,\mathbf Q_{t+s})\le c_t|s|,\quad |s|\le s_t.
\end{equation}
(The proof of the corresponding inequality for $\delta(\mathbf
Q_{t+s}, \mathbf Q_t)$ can be carried out in a similar way.)

Fix an arbitrary number $t$ and take a function $u\in\Dom(\mathbf
Q_t)$. According to the definition~\eqref{eqDeltaTT+S}, it
suffices to find a function $v_s\in\Dom(\mathbf Q_{t+s})$ (which
depends on $u$) such that
\begin{equation}\label{lSmallGapTT+S23}
\|u-v_s\|_{L_2(G)}+\|\mathbf A u-\mathbf Av_s\|_{H_a^0(G)}\le
k_1|s|\cdot\|(u,\mathbf Au)\|,
\end{equation}
where $|s|$ is sufficiently small and ${k_1,k_2,\dots}>0$ may
depend on $t$ but do not depend on $u$ and $s$.

Let us search $v_s\in\Dom(\mathbf Q_{t+s})$ in the form
\begin{equation}\label{lSmallGapTT+S24}
v_s=u+w_s,
\end{equation}
where $w_s\in H_a^{2m}(G)$ is a solution of the problem
\begin{equation}\label{lSmallGapTT+S25}
\mathbf Aw_s=\sum\limits_{j=1}^{J_s}\beta_j^s f_j^s,\qquad
(\mathbf B_{i\mu}^0+\mathbf B_{i\mu}^1+(t+s)\mathbf
C_{i\mu})w_s=-s\mathbf C_{i\mu}u;
\end{equation}
the numbers $J_s$ and $\beta_j^s$ as well as the functions
$f_j^s\in H_a^0(G)$ will be defined later in such a way that the
solution $w_s\in H_a^{2m}(G)$ exists.

2. To solve problem~\eqref{lSmallGapTT+S25}, we first note that
$\mathbf C_{i\mu}u\in H_a^{2m-m_{i\mu}-1/2}(\Gamma_i)$ due to
Corollary~\ref{corAprEstCimu}. Hence, we can apply
Lemma~\ref{lHomogeneous} and construct a function $W_s\in
H_a^{2m}(G)$ such that
\begin{equation}\label{lSmallGapTT+S27}
(\mathbf B_{i\mu}^0+\mathbf B_{i\mu}^1+(t+s)\mathbf
C_{i\mu})W_s=-s\mathbf C_{i\mu}u,
\end{equation}
\begin{equation}\label{lSmallGapTT+S28}
\|W_s\|_{H_a^{2m}(G)}\le k_2|s| \sum\limits_{i,\mu}\|\mathbf
C_{i\mu}u\|_{H_a^{2m-m_{i\mu}-1/2}(\Gamma_i)}.
\end{equation}
Combining~\eqref{lSmallGapTT+S28} with~\eqref{eqAprEstCimu}, we
obtain
\begin{equation}\label{lSmallGapTT+S29'}
\|W_s\|_{H_a^{2m}(G)}\le k_3|s|\cdot\|(u,\mathbf Au)\|.
\end{equation}
Clearly, problem~\eqref{lSmallGapTT+S25} is equivalent to the
following one:
\begin{equation}\label{lSmallGapTT+S29}
\mathbf AY_s=-\mathbf AW_s+\sum\limits_{j=1}^{J_s}\beta_j^s
f_j^s,\qquad (\mathbf B_{i\mu}^0+\mathbf B_{i\mu}^1+(t+s)\mathbf
C_{i\mu})Y_s=0,
\end{equation}
where
\begin{equation}\label{lSmallGapTT+S31}
Y_s=w_s-W_s\in H_a^{2m}(G).
\end{equation}

3. To solve problem~\eqref{lSmallGapTT+S29}, we consider the
bounded operator
\begin{equation}\label{eqLt}
\mathbf L_t=\{{\bf A},\mathbf B_{i\mu}^0+\mathbf
B_{i\mu}^1+t\mathbf C_{i\mu}\}: H_a^{2m}(G)\to \mathcal
H_a^0(G,\Gamma).
\end{equation}
Note that $\mathbf C_{i\mu}^2 v\in
H_a^{2m-m_{i\mu}-1/2}(\Gamma_i)$ for any $v\in H_a^{2m}(G)$ due to
Lemma~\ref{lCinHa}; for this reason, we can write $\mathcal
H_a^0(G,\Gamma)$ instead of $\mathcal H_a^0(G,\Gamma)\dotplus
\mathcal R_a^0(G,\Gamma)$ in the definition of the operator
$\mathbf L_t$ (cf.~\eqref{eqOpL}). It follows from Theorem~6.1
in~\cite{GurRJMP04} and from Remark~\ref{remOneAndTheSameL} that
the operator $\mathbf L_t$ has the Fredholm property for any
$t\in\mathbb C$.

Expand the space $H_a^{2m}(G)$ into the orthogonal sum $
H_a^{2m}(G)=\ker\mathbf L_t\oplus E_t, $ where $E_t$ is a closed
subspace in $H_a^{2m}(G)$. Clearly, the operator
\begin{equation}\label{eqLt'}
\mathbf L_t'=\{{\bf A},\mathbf B_{i\mu}^0+\mathbf
B_{i\mu}^1+t\mathbf C_{i\mu}\}: E_t\to \mathcal H_a^0(G,\Gamma)
\end{equation}
has the Fredholm property and its kernel is trivial. In
particular, this means that
\begin{equation}\label{lSmallGapTT+S32}
\|u\|_{H_a^{2m}(G)}\le k_4\|\mathbf L_t'u\|_{\mathcal
H_a^0(G,\Gamma)}\quad \forall u\in E_t.
\end{equation}

Let $ J=\codim\mathcal R(\mathbf L_t'). $ It follows from
Lemma~\ref{lCinHa} and from Theorem~\ref{thAppBTheorSec16Kr} that
the operator
$$
\mathbf L_{ts}'=\{{\bf A},\mathbf B_{i\mu}^0+\mathbf
B_{i\mu}^1+(t+s)\mathbf C_{i\mu}\}: E_t\to \mathcal
H_a^0(G,\Gamma)
$$
also has the Fredholm property, its kernel is trivial and
$\codim\mathcal R(\mathbf L_{ts}')=J$, provided that $|s|\le s_t$,
where $s_t>0$ is sufficiently small. Moreover, using
estimates~\eqref{lSmallGapTT+S32}, \eqref{lCinHa_1},
and~\eqref{lCinHa_2}, we have, for all $u\in E_t$,
$$
\|u\|_{H_a^{2m}(G)}\le k_4\left(\|\mathbf L_{ts}'u\|_{\mathcal
H_a^0(G,\Gamma)}+s_t\sum\limits_{i,\mu}\|\mathbf
C_{i\mu}u\|_{H_a^{2m-m_{i\mu}-1/2}(\Gamma_i)}\right)\le k_5
\left(\|\mathbf L_{ts}'u\|_{\mathcal
H_a^0(G,\Gamma)}+s_t\|u\|_{H_a^{2m}(G)}\right).
$$
Taking $s_t\le 1/(2k_6)$, we obtain
\begin{equation}\label{lSmallGapTT+S33}
\|u\|_{H_a^{2m}(G)}\le k_6\|\mathbf L_{ts}'u\|_{\mathcal
H_a^0(G,\Gamma)}\quad \forall u\in E_t.
\end{equation}

Since $\mathbf L_{ts}'$ has the Fredholm property, the set $\{f\in
H_a^0(G):\ (f,0)\in\mathcal R(\mathbf L_{ts}')\}$ is closed and is
of finite codimension $J_s$ in $H_a^0(G)$. It is easy to see that
$ J_s\le J. $

Let $f_1^s,\dots,f_{J_s}^s$ be an orthogonal normalized basis for
the space
$$
H_a^0(G)\ominus\{f\in H_a^0(G):\ (f,0)\in\mathcal R(\mathbf
L_{ts}')\}.
$$

Set $\beta_j^s=(\mathbf A W_s,f_j^s)_{H_a^0(G)}$. In this case,
problem~\eqref{lSmallGapTT+S29} admits a unique solution $Y_s\in
E_t$, and, by virtue of~\eqref{lSmallGapTT+S33}
and~\eqref{lSmallGapTT+S29'}, we have
\begin{equation}\label{lSmallGapTT+S35}
\|Y_s\|_{H_a^{2m}(G)}\le k_6\left(\|\mathbf A
W_s\|_{H_a^0(G)}+\sum\limits_{j=1}^{J_s}|\beta_{j}^s|\right)\le
k_7|s|\cdot\|(u,\mathbf Au)\|+k_6
J\max\{\beta_1^s,\dots,\beta_{J_s}^s\}.
\end{equation}

Estimating $\beta_{j}^s=(\mathbf A W_s,f_j^s)_{H_a^0(G)}$ by
Schwarz' inequality and using~\eqref{lSmallGapTT+S29'}, we obtain
$$
|\beta_{j}^s|\le\|\mathbf A W_s\|_{H_a^0(G)}\le
k_8|s|\cdot\|(u,\mathbf Au)\|.
$$
Combining this inequality with~\eqref{lSmallGapTT+S35} yields
\begin{equation}\label{lSmallGapTT+S36}
\|Y_s\|_{H_a^{2m}(G)}\le k_9|s|\cdot\|(u,\mathbf Au)\|.
\end{equation}

4. Taking into account equality~\eqref{lSmallGapTT+S31}, we deduce
from estimates~\eqref{lSmallGapTT+S29'}
and~\eqref{lSmallGapTT+S36}
\begin{equation}\label{lSmallGapTT+S37}
\|w_s\|_{L_2(G)}\le k_{10}\|w_s\|_{H_a^{2m}(G)}\le k_{11}|s|\cdot
\|(u,\mathbf Au)\|,
\end{equation}
\begin{equation}\label{lSmallGapTT+S38}
\|\mathbf A w_s\|_{H_a^{0}(G)}\le k_{12}\|w_s\|_{H_a^{2m}(G)}\le
k_{12}k_{11}|s|\cdot \|(u,\mathbf Au)\|,
\end{equation}
where $w_s=Y_s+W_s$ is a solution of
problem~\eqref{lSmallGapTT+S25}.

It follows from the boundedness of the embedding
$H_a^{2m}(G)\subset W^m(G)$ that the function $v_s$ defined
by~\eqref{lSmallGapTT+S24} belongs to $W^m(G)$, and
$v_s\in\Dom(\mathbf Q_{t+s})$ due to the second relation
in~\eqref{lSmallGapTT+S25}. The desired
inequality~\eqref{lSmallGapTT+S23} follows
from~\eqref{lSmallGapTT+S24}, \eqref{lSmallGapTT+S37},
and~\eqref{lSmallGapTT+S38}.
\end{proof}

\begin{proof}[Proof of Theorem~{\rm \ref{thIndP0=IndP1}}]
It follows from Lemma~2.1 in~\cite{SkDu90} that the spectrum of
$\tilde{\mathcal L}(\lambda)$ is discrete. Therefore, one can find
a number $a$ such that $m-1<a<m$ and the lines $\Im\lambda=a+1-2m$
and $\Im\lambda=a+1-m$ contain no eigenvalues of $\tilde{\mathcal
L}(\lambda)$. Fix two arbitrary numbers $t_1,t_2\in\mathbb C$. By
Lemma~\ref{lIndQ=IndP} and Remark~\ref{remOneAndTheSameL}, the
operators $\mathbf Q_t$ have the Fredholm property for all $t$ in
the interval $I_{t_1t_2}\subset \mathbb C$ with the end points
$t_1,t_2$. Covering each point of the interval $I_{t_1t_2}$ by a
disk of sufficiently small radius, choosing a finite subcovering
of $I_{t_1t_2}$, and applying Theorems~\ref{thSmallGap}
and~\ref{thAppBTheor5.17Kato}, we see that $\ind \mathbf
Q_{t_1}=\ind \mathbf Q_{t_2}$. It follows from this fact and from
Lemma~\ref{lIndQ=IndP} that $\ind \mathbf P_{t_1}=\ind \mathbf
P_{t_2}$.
\end{proof}

\begin{remark}\label{remManyOrbits}
Theorems~\ref{thIndP=IndP'} and~\ref{thIndP0=IndP1} remain true in
the case where the set $\mathcal K$ consists of finitely many
disjoint orbits. The proofs need evident modifications.
\end{remark}

\appendix

\section{Appendix}

\subsection{Some Properties of Sobolev and Weighted
Spaces}\label{secPropSobSpace}

Let $G$ and $\Gamma_i$ be the same as in Sec.~\ref{secSetting}.

\begin{lemma}[see Lemma~3.5 in~\cite{KondrTMMO67}]\label{lAppALem3.5Kondr}
Let $k_2>k_1$ and $k_2-a_2> k_1-a_1$. Then the space
$H_{a_2}^{k_2}(G)$ is compactly embedded into $H_{a_1}^{k_1}(G)$.
\end{lemma}

Fix an arbitrary index $i$ and set $\Gamma=\Gamma_i$. Let
$g\in\overline{\Gamma}\setminus\Gamma$. We assume throughout this
section, without loss of generality, that $g=0$ and $\Gamma$
coincides with the axis $Oy_1$ in a sufficiently small
neighborhood $\mathcal O_\varepsilon(0)$ of the origin. In this
section, we use the notation
$$
G^\varepsilon=G\cap\mathcal O_\varepsilon(0),\qquad
\Gamma^\varepsilon=\Gamma\cap\mathcal O_\varepsilon(0),
$$
in which case $H_a^k(G^\varepsilon)=H_a^k(G^\varepsilon,\{0\})$.

\begin{lemma}\label{lu=P+v}
If $u\in W^k(G^\varepsilon)$, $k\ge1$, then the following
assertions are true{\rm :}
\begin{enumerate}
\item
$u(y)=P(y)+v(y)$ for $y\in G^\varepsilon,$ where $
P(y)=\sum\limits_{|\alpha|\le k-2}p_\alpha y^\alpha$, $v\in
W^k(G^\varepsilon)\cap H_\delta^k(G^\varepsilon)$ $\forall\delta>0
$ {\rm(}if $k=1$, we set $P(y)\equiv0${\rm);} in particular, $u\in
H_{k-1+\delta}^{k}(G^\varepsilon)${\rm;}
\item
$D^\alpha u|_{y=0}=D^\alpha P|_{y=0}$ for $|\alpha|\le k-2;$
\item $
\sum\limits_{|\alpha|\le
k-2}|p_\alpha|+\|v\|_{H_\delta^k(G^\varepsilon)}\le
c_\delta\|u\|_{W^k(G^\varepsilon)}, $ where $c_\delta>0$ does not
depend on $u$.
\end{enumerate}
\end{lemma}
{\it Proof} follows from Lemma~4.9 in~\cite{KondrTMMO67} for $k=1$
and from Lemma~4.11 in~\cite{KondrTMMO67} for $k\ge2$.

\begin{lemma}\label{lPsi=P1+Phi}
If $\psi\in W^{k-1/2}(\Gamma^\varepsilon)$, $k\ge1$, then the
following assertions are true{\rm :}
\begin{enumerate}
\item
$\psi(r)=P_1(r)+\varphi(r)$ for $0<r<\varepsilon,$ where $
P_1(r)=\sum\limits_{\beta=0}^{k-2}p_\beta r^\beta$, $\varphi\in
W^{k-1/2}(\Gamma^\varepsilon)\cap
H_\delta^{k-1/2}(\Gamma^\varepsilon)$ $\forall\delta>0 $ {\rm(}if
$k=1$, we set $P_1(r)\equiv0${\rm);} in particular, $\psi\in
H_{k-1+\delta}^{k-1/2}(\Gamma^\varepsilon)${\rm;}
\item
$(d^\beta \psi/d r^\beta)|_{r=0}=(d^\beta P_1/d r^\beta)|_{r=0}$
for $\beta=0,\dots, k-2;$
\item $
\sum\limits_{\beta=0}^{k-2}|p_\beta|+\|\varphi\|_{H_\delta^{k-1/2}(\Gamma^\varepsilon)}\le
c_\delta\|\psi\|_{W^{k-1/2}(\Gamma^\varepsilon)}, $ where
$c_\delta>0$ does not depend on $\psi$.
\end{enumerate}
\end{lemma}
\begin{proof}
Consider a function $u\in W^k(G^\varepsilon)$ such that
$u|_{\Gamma^\varepsilon}=\psi$ and $ \|u\|_{W^k(G^\varepsilon)}\le
2\|\psi\|_{W^{k-1/2}(\Gamma^\varepsilon)}. $ Now it remains to
apply Lemma~\ref{lu=P+v}.
\end{proof}

\begin{lemma}\label{lPsig=0}
Let $\psi\in W^{k-1/2}(\Gamma)$, $k\ge2$, and let
\begin{equation}\label{eqPsig=0_1}
\dfrac{d^s\psi}{dr^s}\bigg|_{y=0}=0,\qquad s=0,\dots, l,
\end{equation}
for a fixed $l\le k-2$. Then $\psi\in
H_{k-2-l+\delta}^{k-1/2}(\Gamma)$ $\forall\delta>0$ and
\begin{equation}\label{eqlPsig=0_2}
\|\psi\|_{H_{k-2-l+\delta}^{k-1/2}(\Gamma)}\le
c_\delta\|\psi\|_{W^{k-1/2}(\Gamma)},
\end{equation}
where $c_\delta>0$ does not depend on $\psi$.
\end{lemma}
\begin{proof}
It follows from relations~\eqref{eqPsig=0_1} and from
Lemma~\ref{lPsi=P1+Phi} (parts~1 and~2) that
\begin{equation}\label{eqlPsig=0_3}
\psi(r)=\sum\limits_{\beta=l+1}^{k-2}p_\beta
r^\beta+\varphi(r),\qquad 0<r<\varepsilon,
\end{equation}
where
\begin{equation}\label{eqlPsig=0_4}
\varphi\in H_\delta^{k-1/2}(\Gamma^\varepsilon)\subset
H_{k-2-l+\delta}^{k-1/2}(\Gamma^\varepsilon),\qquad \delta>0.
\end{equation}

If $l=k-2$, then the sum in~\eqref{eqlPsig=0_3} is absent and the
lemma follows from~\eqref{eqlPsig=0_4} and from part~3 of
Lemma~\ref{lPsi=P1+Phi}.

If $l\le k-3$, then the sum comprises the terms $r^\beta$ for
$\beta\ge l+1$. One can directly verify that $r^\beta\in
H_{k-2-l+\delta}^{k-1/2}(\Gamma^\varepsilon)$ for the above
$\beta$ and $\forall\delta>0$. Therefore,
combining~\eqref{eqlPsig=0_3} with~\eqref{eqlPsig=0_4} and with
part~3 of Lemma~\ref{lPsi=P1+Phi}, we complete the proof.
\end{proof}

\begin{lemma}\label{lbPsi}
Let $\psi\in H_{a+l}^{k-1/2}(\Gamma)$, $l,k\in\mathbb N$,
$a\in\mathbb R$, and let $b\in C^\infty(\overline{\Gamma})$ be a
compactly supported function satisfying the relations $
\dfrac{\partial^s b}{\partial r^s}\bigg|_{r=0}=0$, $ s=0,\dots
l-1. $ Then
\begin{equation}\label{eqbPsi_A}
\|b\psi\|_{H_a^{k-1/2}(\Gamma)}\le
c\|\psi\|_{H_{a+l}^{k-1/2}(\Gamma)}.
\end{equation}
\end{lemma}
\begin{proof}
Clearly, it suffices to carry out the proof for compactly
supported functions $\psi$ and for $Q$ and $\Gamma$ replaced by
$K=\{y\in\mathbb R^2:\ 0<\omega<\omega_0\}$ and
$\gamma=\{y\in\mathbb R^2:\ \omega=0\}$, respectively.

 Denote by $\hat b\in C^\infty(\mathbb R)$ an extension of
$b(y_1)$ to $\mathbb R$ and introduce the function
$B(y_1,y_2)=\hat b(y_1)$ for $(y_1,y_2)\in\mathbb R^2$. Clearly,
we have
\begin{equation}\label{eqbPsi_B}
B\in C^\infty(\overline{K}),\qquad D^\sigma B|_{y=0}=0,\quad
|\sigma|\le l-1.
\end{equation}

Let $u\in H_{a+l}^k(K)$ be a compactly supported extension of
$\psi$ to the angle $K$ such that
\begin{equation}\label{eqbPsi_C}
\|u\|_{H_{a+l}^k(K)}\le c_1\|\psi\|_{H_{a+l}^{k-1/2}(\gamma)}.
\end{equation}
It follows from Teylor's formula and from~\eqref{eqbPsi_B} that $
|D^\sigma B|=O\big(r^{l-|\sigma|}\big) $ for any $\sigma$;
therefore,
\begin{multline*}
\|Bu\|_{H_{a}^k(K)}^2=\sum\limits_{|\alpha|\le k}\int\limits_K
r^{2(a+|\alpha|-k)}|D^\alpha(Bu)|^2 dy\le c_2
\sum\limits_{|\sigma|+|\zeta|\le k}\int\limits_K
r^{2(a+|\sigma|+|\zeta|-k)}|D^\sigma B|^2|D^\zeta u|^2 dy\\
\le c_3 \sum\limits_{|\zeta|\le k}\int\limits_K
r^{2(a+l+|\zeta|-k)}|D^\zeta u|^2 dy=c_3\|u\|_{H_{a+l}^k(K)}^2
\end{multline*}
(remind that $u$ is compactly supported). Combining this estimate
with~\eqref{eqbPsi_C}, we finally obtain
$$
\|b\psi\|_{H_{a}^{k-1/2}(\gamma)}\le\|Bu\|_{H_{a}^k(K)}\le
c_3^{1/2}\|u\|_{H_{a+l}^k(K)}\le
c_3^{1/2}c_1\|\psi\|_{H_{a+l}^{k-1/2}(\gamma)}.
$$
\end{proof}

\subsection{Some Properties of Fredholm Operators}

Let $H_1$ and $H_2$ be Hilbert spaces, and let $ P:\Dom(P)\subset
H_1\to H_2 $ be a linear (in general, unbounded) operator.

\begin{definition}\label{defFredholm}
The operator $P$ is said {\it to have the Fredholm property} if it
is closed, its image is closed, and the dimension of its kernel
$\ker P$ and the codimension of its image $\mathcal R(P)$ are
finite. The number $\ind P=\dim\ker P-\codim\mathcal R(P)$ is
called an {\it index} of the Fredholm operator $P$.
\end{definition}

\begin{theorem}[see Theorem~7.1 in~\cite{Kr}]\label{thAppBLem7.1Kr}
Let $H$ be a Hilbert space such that $H_1$ is compactly embedded
into $H$, and let the operator $P$ be closed. Then $\dim\ker
P<\infty$ and $\mathcal R(P)=\overline{\mathcal R(P)}$ iff
$$
\|u\|_{H_1}\le c(\|Pu\|_{H_2}+\|u\|_H)\quad \forall u\in\Dom(P).
$$
\end{theorem}

The proof of the following result is contained in part~2 of the
proof of Lemma~2.5 in~\cite{GurMatZam05}.
\begin{theorem}\label{thAppBdotPP}
Let $ \dot P:\Dom(\dot P)\subset H_1\to H_2 $ be a Fredholm
operator such that $P$ is an extension of $\dot P$, i.e., $\dot
P\subset P.$ Suppose that $\dim\ker P<\infty$, $\mathcal
R(P)=\overline{\mathcal R(P)}$, and $\codim \mathcal R(P)<\infty$.
Then the operator $P$ is closed {\rm (}hence, it has the Fredholm
property{\rm )}.
\end{theorem}

Let $ A:\Dom(A)\subset H_1\to H_2 $ be a linear operator.

\begin{theorem}[see Sec.~16 in~\cite{Kr}]\label{thAppBTheorSec16Kr} Let the
operator $P$ have the Fredholm property, $A$ be bounded, and
$\Dom(A)=H_1$. Then the operator $P+A$ has the Fredholm property,
$\ind(P+A)=\ind P$, $\dim\ker(P+A)\le\dim\ker P$, and
$\codim\mathcal R(P+A)\le \codim\mathcal R(P)$, provided that
$\|A\|$ is sufficiently small.
\end{theorem}

\begin{definition}[see,
e.g.,~\cite{Kr,Kato}]\label{defPCompact} The operator $A$ is said
to be {\em relatively compact with respect to $P$} or simply {\em
$P$-compact} if $\Dom(P)\subset\Dom(A)$ and, for any sequence
$u_n\in\Dom(P)$ with both $\{u_n\}$ and $\{Pu_n\}$ bounded,
$\{Au_n\}$ contains a convergent subsequence.
\end{definition}

\begin{theorem}[see Theorem~5.26
in Chap.~4 of~\cite{Kato}]\label{thAppBTheor5.26Kato} Suppose that
the operator $P$ has the Fredholm property and the operator $A$ is
$P$-compact. Then the operator $P+A$ also has the Fredholm
property and $\ind (P+A)=\ind P$.
\end{theorem}

Finally, we introduce a concept of a gap between closed operators.
Let $ S:\Dom(S)\subset H_1\to H_2 $ be a linear operator. In the
space $H_1\times H_2$, we introduce the norm
$$
\|(u,f)\|=\left(\|u\|_{H_1}^2+\|f\|_{H_2}^2\right)^{1/2}\quad
\forall (u,f)\in H_1\times H_2.
$$

Set $ \delta(P,S)=\sup\limits_{u\in\Dom(P):\ \|(u,
Pu)\|=1}\dist\big((u, Pu),\Gr S\big)$, where $\Gr S$ is the graph
of the operator $S$.
\begin{definition}\label{defGap}
The number $ \hat\delta(P,S)=\max\{\delta(P,S),\delta(S,P)\} $ is
called a {\it gap between the operators $P$ and $S$}.
\end{definition}

\begin{theorem}[see Theorem~5.17 in Chap.~4
of~\cite{Kato}]\label{thAppBTheor5.17Kato} Let the operator $P$
have the Fredholm property and $S$ be closed. Then the operator
$S$ has the Fredholm property, $\ind S=\ind P$, $\dim\ker
S\le\dim\ker P$, and $\codim\mathcal R(S)\le \codim\mathcal R(P)$,
provided that the gap $\hat\delta(P,S)$ is sufficiently small.
\end{theorem}

The author is grateful to A.~L.~Skubachevskii for attention.

\end{document}